\documentclass[11pt,english,reqno]{amsart}

\usepackage{layout}
\usepackage{amssymb}
\usepackage{amsfonts}
\usepackage{amsmath,amscd,amsthm}
\usepackage{newlfont}
\usepackage[dvips]{graphicx}
\usepackage{epsf,epsfig}
\usepackage{color}
\usepackage[latin1]{inputenc}
\usepackage{hhline}
\usepackage{rotating}
\usepackage{verbatim}

\usepackage[footnotesize]{subfigure}
\usepackage{caption}
\usepackage{picture}
\usepackage{pstcol,pst-fill,pst-grad}

\newcommand{\R}{\mathbb R}

\newcommand{\N}{\mathbb N}
\newcommand{\Z}{\mathbb Z}
\newcommand{\Q}{\mathbb Q}
\renewcommand{\P}{\mathbb P}
\newcommand{\E}{\mathbb E}

\newcommand{\indic}{1\negthickspace\text{I}}
\newcommand{\indicat}{\text{1}\negthickspace\text{I}}

\newcommand{\dd}{\text{d}}

\newcommand{\pass}{\vspace{0.3cm}\noindent}

\newcommand{\sk}{\operatorname{Sk}}
\DeclareMathOperator{\dimbox}{\dim_B}\DeclareMathOperator{\dimhaus}{\dim_H}
 
\DeclareMathOperator{\osc}{osc} 
 
 \DeclareMathOperator{\dist}{dist}
\DeclareMathOperator{\vol}{Vol}

\theoremstyle{plain}
\newtheorem{theo}{Theorem}[section]
\newtheorem{prop}[theo]{Proposition}
\newtheorem{lem}[theo]{Lemma}

\newtheorem{theor}{Theorem}

\theoremstyle{definition}

\definecolor{note}{rgb}{0.2,0.1,1}

\definecolor{coupe}{rgb}{1,0.1,0.6}

\newcommand{\goodgap}{\hspace{\subfigcapskip}}

\author[Pierre Calka \& Yann Demichel]{Pierre Calka\textsuperscript{1} and Yann Demichel\textsuperscript{2}}

\address{\textsuperscript{1} Laboratoire de Mathématiques Raphaël Salem, UMR 6085, Université de Rouen, avenue de
l'Université, Technopôle du Madrillet, 76801 Saint-Etienne-du-Rouvray, France.}
\email{pierre.calka@univ-rouen.fr}%
\address{\textsuperscript{2} Laboratoire MODAL'X, EA 3454, Université Paris Ouest Nanterre La Défense, 200 avenue de la Ré\-pu\-bli\-que, 92001 Nanterre, France.}
\email{yann.demichel@u-paris10.fr}%

\address{This work was partially supported by the French ANR grant PRESAGE (ANR-11-BS02-003), the French ANR grant MATAIM (ANR-09-BLAN-0029-01) and the French research group GeoSto (CNRS-GDR3477).}

\title[Fractal random series generated by Poisson-Voronoi tessellations]{Fractal random series generated by Poisson-Voronoi tessellations}

\begin{document}

\begin{abstract}
In this paper, we construct a new family of random series defined on $\R^D$, indexed by one scaling parameter and two Hurst-like exponents. The model is close to Takagi-Knopp functions, save for the fact that the underlying partitions of $\R^D$ are not the usual dyadic meshes but random Voronoi tessellations generated by Poisson point processes. This approach leads us to a continuous function whose random graph is shown to be fractal with explicit and equal box and Hausdorff dimensions. The proof of this main result is based on several new distributional properties of the Poisson-Voronoi tessellation on the one hand, an estimate of the oscillations of the function coupled with an application of a Frostman-type lemma on the other hand. Finally, we introduce two related models and provide in particular a box-dimension calculation for a derived deterministic Takagi-Knopp series with hexagonal bases.
\end{abstract}




\maketitle

\fussy
\section*{Introduction}\label{sec:intro}

\noindent The original Weierstrass series (see \cite{wei95}) is a fundamental example of continuous but nowhere differentiable function.
Among the more general family of Weierstrass-type functions, the Takagi-Knopp series can be defined in one dimension as
\begin{equation}\label{def:taka1D}
K_H(x)= \sum_{n=0}^{\infty} 2^{-nH} \Delta(2^nx)\,\,,\,x\in\R,
\end{equation}
where $\Delta(x)=\dist(x,\Z)$ is the sawtooth -or pyramidal- function and $H\in(0,1]$ is called the Hurst parameter of the function. Introduced at the early beginning of the 20th century (see \cite{tak03,knopp18}), they have been extensively studied since then (see the two recent surveys \cite{all11,laga}).

\pass The construction of $K_H$ is only based on two ingredients: a sequence of partitions of $\R$ (the dyadic meshes) associated with a decreasing sequence of amplitudes for the consecutive layers of pyramids. Therefore, we can easily extend definition \eqref{def:taka1D} to dimension $D\geqslant2$ using the $D$-dimensional dyadic meshes.

\pass In order to provide realistic models for highly irregular signals such as rough surfaces (see \cite{four,mill,dub89} and Chapter 6 in \cite{russ}), it is needed to randomize such deterministic functions. Two common ways to do it are the following: either the pyramids are translated at each step by a random vector (see e.g. \cite{tri95,heurt03,demtri06}), or the height of each pyramid is randomly chosen (see in particular \cite{four} for the famous construction of the Brownian bridge).

\pass In many cases the graphs of such functions are fractal sets. Therefore, their fractal dimensions provide crucial information for describing the roughness of the data (see \cite{mand84,issa03}). The two most common fractal dimensions are the box-dimension and the Hausdorff dimension. The former is in general easier to calculate whereas the latter is known only in very special cases (see e.g. \cite{maul86,led,hunt98,car,bar11}).

\pass In this paper, a new family of Takagi-Knopp type series is introduced. Contrary to the previous randomization procedures, our key-idea is to substitute a sequence of random partitions of $\R^D$ for the dyadic meshes. An alternative idea would have been to keep the cubes and choose independently and uniformly in each cube each center of a pyramid. Notably because the mesh has only $D$ directions, it would be very tricky to calculate the Hausdorff dimension. One advantage for applicational purposes may be to get rid of the rigid structure induced by the cubes and to provide more flexibility with the irregular pattern. A classical model of a random partition is the Poisson-Voronoi tessellation.

\pass For a locally finite set of points called nuclei, we construct the associated Voronoi partition of $\R^D$ by associating to each nucleus $c$ its cell $\cal{C}_c$, i.e. the set of points which are closer to $c$ than to any other nucleus. When the set of nuclei is a homogeneous Poisson point process, we speak of a Poisson-Voronoi tessellation (see e.g. \cite{oka00,mol94,cal10}). In particular, the Poisson point process (resp. the tessellation) is invariant under any measure preserving transformation of $\R^D$, in particular any isometric transformation (see \eqref{eq:invtrans}). Moreover the cells from the tessellation are almost surely convex polytopes. Classical results for the typical Poisson-Voronoi cell include limit theorems (see \cite{ab93}), distributional (see \cite{bl07,cal03}) and asymptotic results (see \cite{hrs04,cs05}). The model is commonly used in various domains such as molecular biology (see \cite{pou04}), thermal conductivity (see \cite{ks95}) or telecommunications (see e.g. \cite{vgs10} and Chapter 5 in \cite{bb09} Volume 1). The only parameter needed to describe the tessellation is the intensity $\lambda>0$, i.e. the mean number of nuclei or cells per unit volume. In particular, the mean area of a typical cell from the tessellation is $\lambda^{-1}$. Multiplied by the scaling factor $\lambda^{\frac1{D}}$, the Poisson-Voronoi tessellation of intensity $\lambda$ is equal in distribution to the Poisson-Voronoi tessellation of intensity one. This scaling invariance is a crucial property that will be widely used in the sequel.

\pass Let $\lambda>1$ and $\alpha,\beta>0$. The parameter $\lambda$ is roughly speaking a scaling factor and $\alpha,\beta$ are Hurst-like exponents. For every integer $n\geqslant0$, we denote by $\cal{X}_n$ a homogeneous Poisson point process of intensity $\lambda^{n\beta}$ in $\R^D$ and by $\cal{T}_n=\{\cal{C}_c:c\in \cal{X}_n\}$ the set of cells of the underlying Poisson-Voronoi tessellation. We recall that
$\lambda^{\frac{n\beta}{D}}\mathcal{T}_n\overset{\mbox{{\tiny{def}}}}{=} \{\lambda^{\frac{n\beta}{D}}\cal{C}_c:c\in \cal{X}_n\}$ is distributed as $\cal{T}_0$ and $\lambda^{\frac{\beta}{D}}{\mathcal T}_n\overset{\mbox{{\tiny law}}}{=}{\mathcal T}_{n-1}$ thanks to the scaling invariance. Moreover, for any isometric transformation $I:\R^D\longrightarrow\R^D$, we have
\begin{equation}\label{eq:invtrans}
I(\mathcal{T}_0)\overset{\mbox{{\tiny{def}}}}{=}\{I(\cal{C}_c):c\in \cal{X}_0\}\overset{\mbox{{\tiny{law}}}}{=}\{\cal{C}_c:c\in \cal{X}_0\}.
\end{equation}

\pass Let $\Delta_n:\R^D \longrightarrow [0,1]$ be the random pyramidal function satisfying $\Delta_n=0$ on $\bigcup_{c\in\cal{X}_n}\partial \cal{C}_c$, $\Delta_n=1$ on $\cal{X}_n$ and piecewise linear (see Figure \ref{fig:mosaicpoiss} and the beginning of section \ref{subsec:randomsets} for more details).

\begin{figure}[!h]
\centering %
\subfigure[The Poisson-Voronoi tessellation ${\mathcal T}_n$.]{\includegraphics[width=7.5cm]{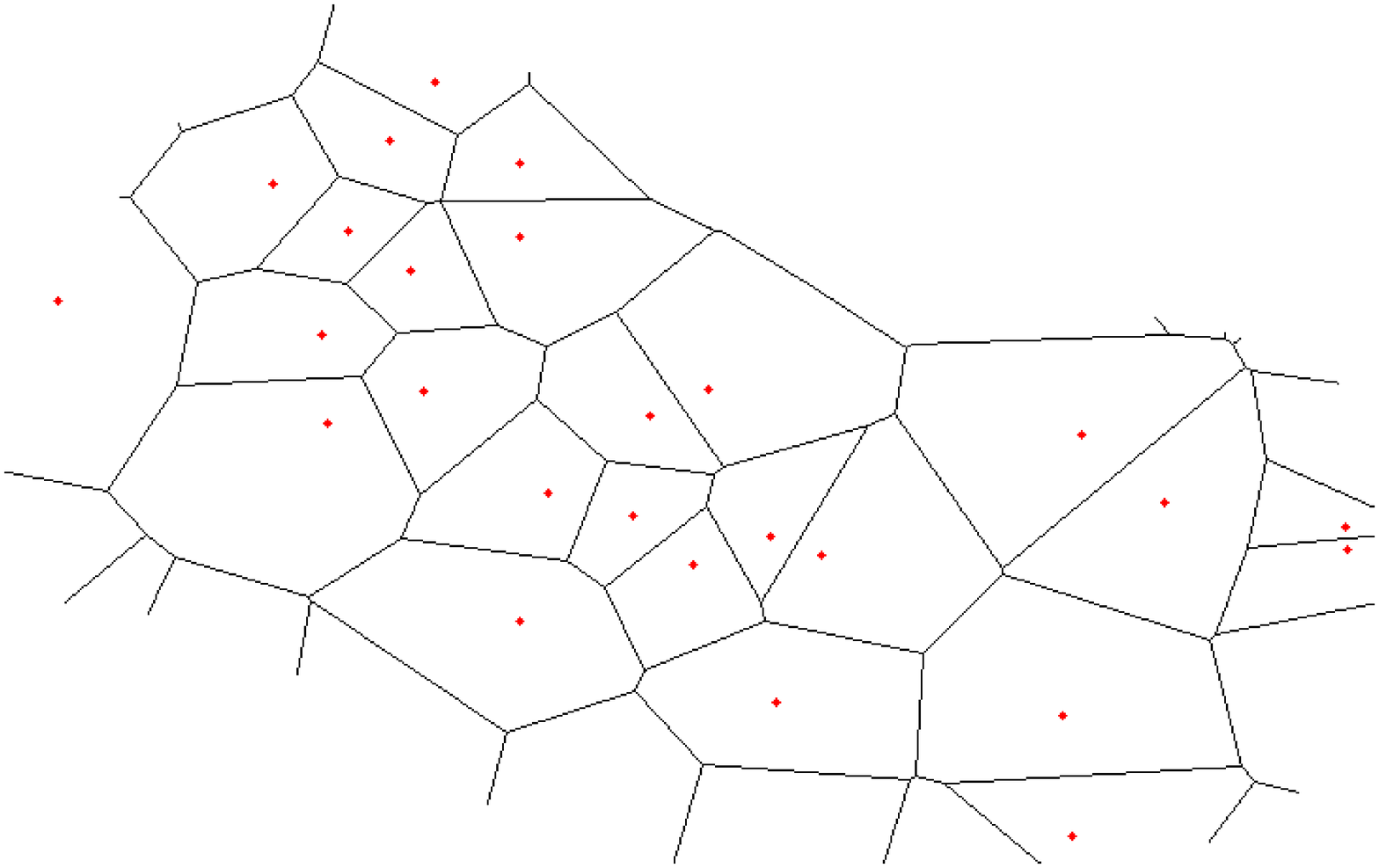}}\goodgap
\subfigure[The pyramidal function $\Delta_n$.]{\includegraphics[width=7.5cm]{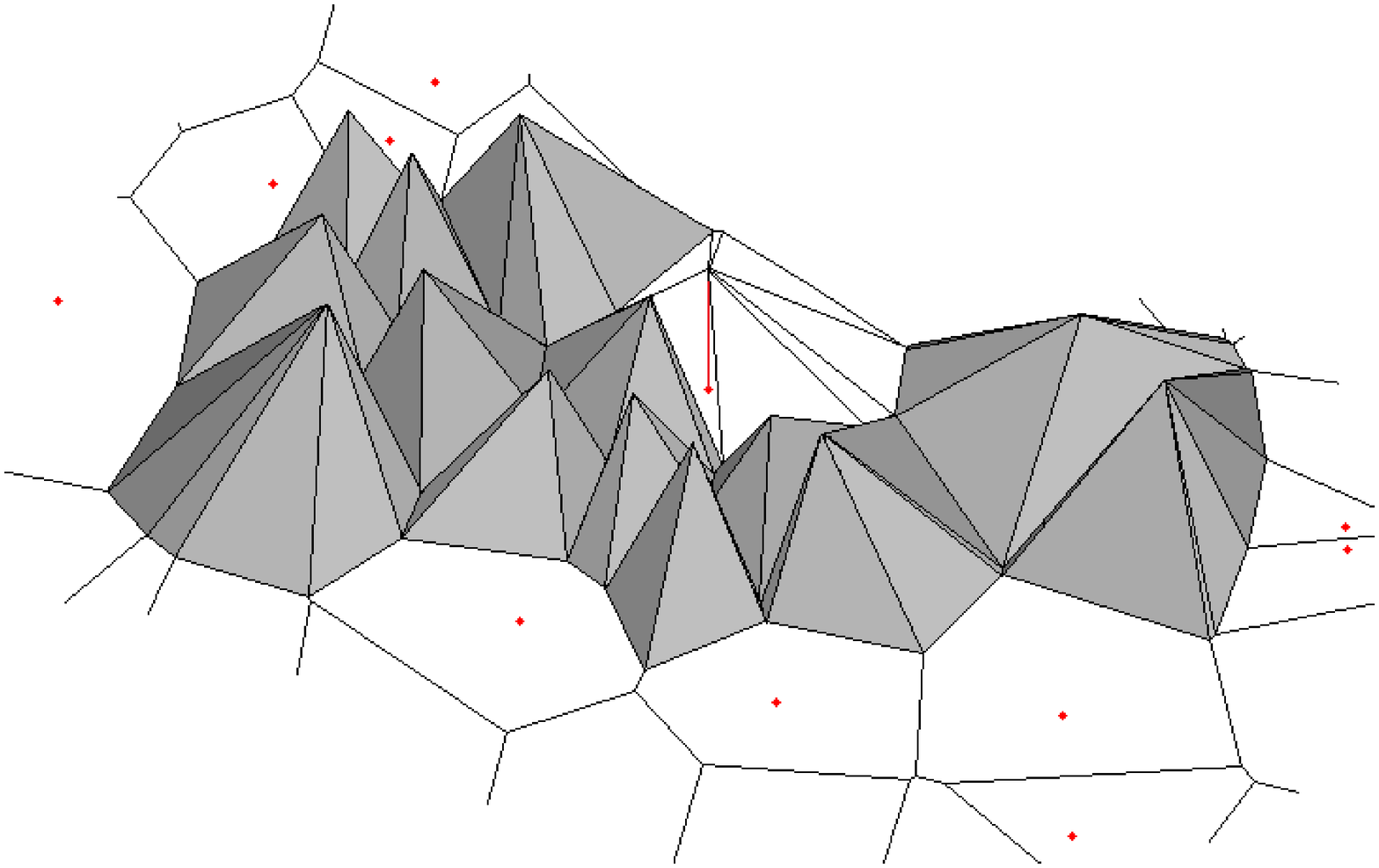}}\goodgap \\
\vspace{-0.2cm}
\caption{\label{fig:mosaicpoiss} Construction of the elementary piecewise linear function $\Delta_n$.}
\end{figure}

\noindent In the sequel the Poisson point processes $\cal{X}_n$ are assumed to be independent. We consider the continuous function
\begin{equation}\label{def:voronoiseries}
F_{\lambda,\alpha,\beta}(x)=\sum_{n=0}^{\infty} \lambda^{-\frac{n\alpha}{D}} \Delta_n(x)\,\,,\,x\in\R^D.
\end{equation}
In particular $F_{\lambda,\alpha,\beta}$ is a sum of independent functions.

\pass Let us denote by $\dimbox(K)$ and $\dimhaus(K)$ the (upper) box-dimension and the Hausdorff dimension of a non-empty compact set $K$ (see e.g. \cite{fal03} for precise definitions). We are mainly interested in the exact values of these dimensions. Our result is the following: \enlargethispage{1cm}

\begin{theor}\label{theo:maintheo}
Let $\lambda>1$ and $0<\alpha\leqslant\beta\leqslant 1$. Then $F_{\lambda,\alpha,\beta}$ is a continuous function whose random graph
\begin{equation*}
\Gamma_{\lambda,\alpha,\beta}=\left\{(x,F_{\lambda,\alpha,\beta}(x)):x\in[0,1]^D\right\}\subset \R^D\times \R
\end{equation*}
is a fractal set satisfying almost surely
\begin{equation}\label{eq:maintheo}
\dimbox(\Gamma_{\lambda,\alpha,\beta})= \dimhaus(\Gamma_{\lambda,\alpha,\beta})=D+1-\frac{\alpha}{\beta}.
\end{equation}
\end{theor}

\noindent Equalities \eqref{eq:maintheo} imply that the smaller $\frac{\alpha}{\beta}$ is, the more irregular $F_{\lambda,\alpha,\beta}$ and $\Gamma_{\lambda,\alpha,\beta}$ are (see Figure \ref{fig:modelpoiss1D} and Figure \ref{fig:modelpoiss2D}).
The result of Theorem \ref{theo:maintheo} naturally holds when $[0,1]^D$ is replaced with any cube of $\R^D$.

\begin{figure}[!h]
\centering
\subfigure[$(\lambda,\alpha,\beta)=(1.2,1,1)$]{\includegraphics[width=6cm]{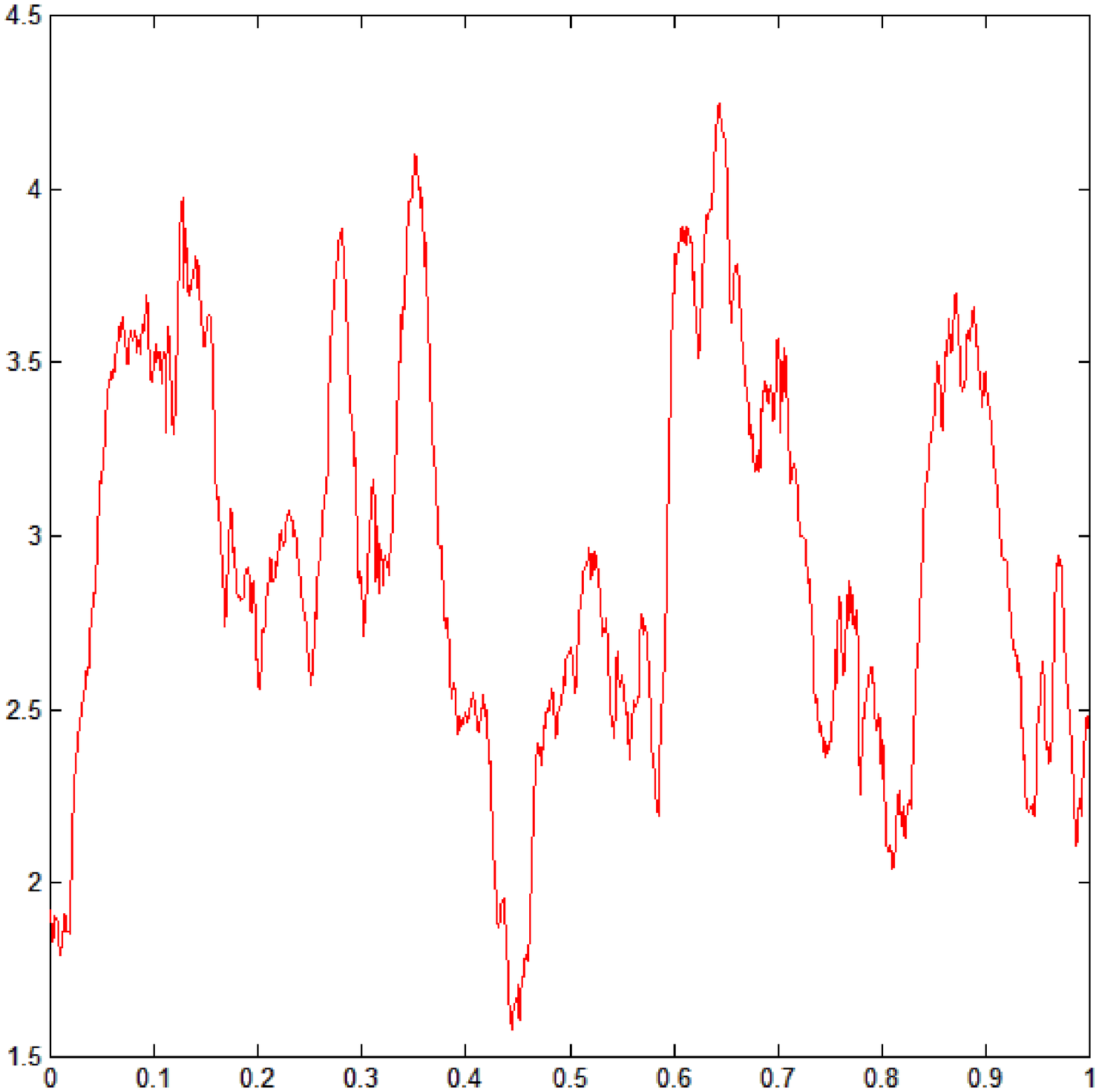}}\goodgap
\subfigure[$(\lambda,\alpha,\beta)=(1.2,0.2,1)$]{\includegraphics[width=6cm]{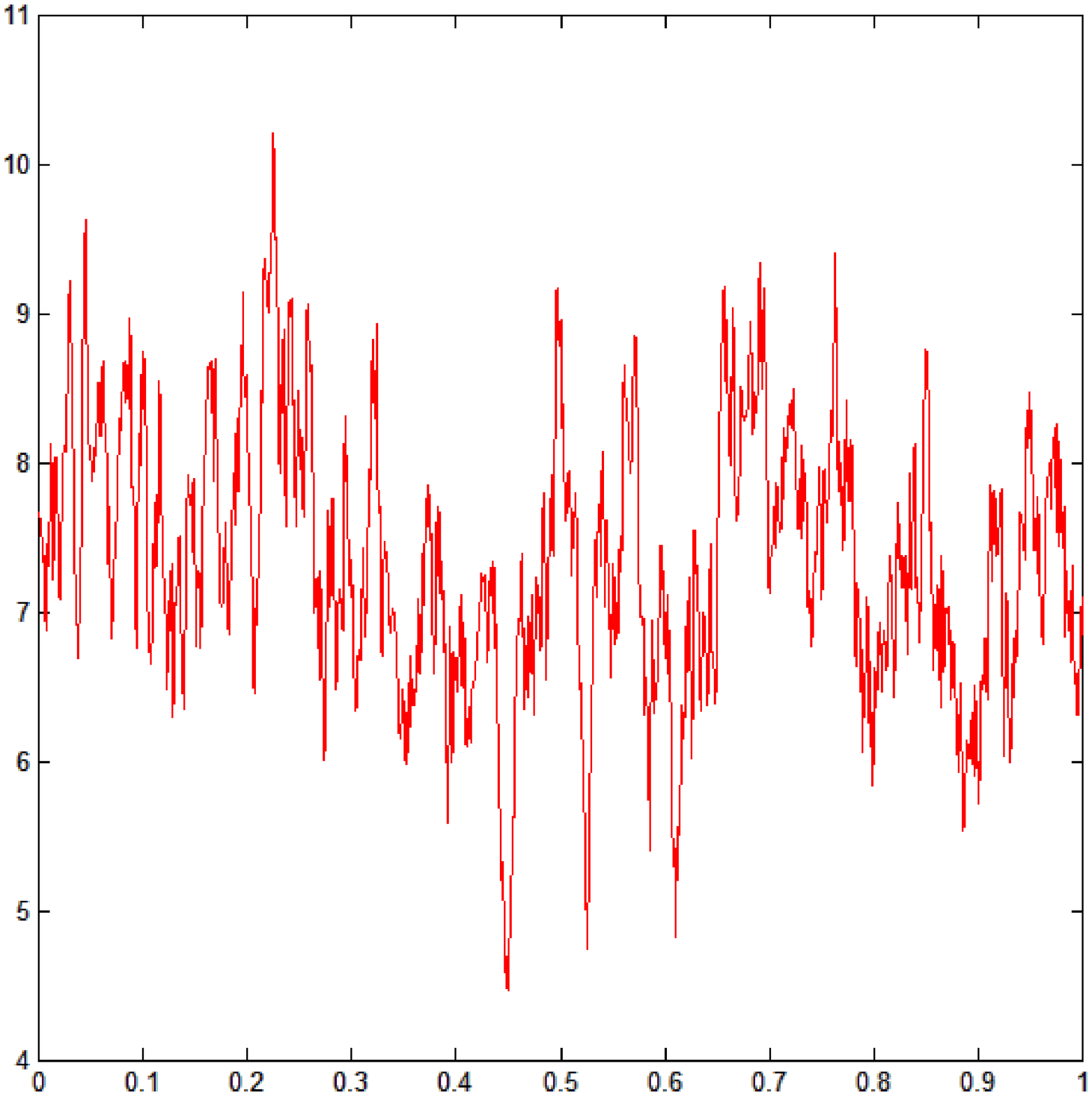}}\goodgap \\
\vspace{-0.2cm}
\caption{\label{fig:modelpoiss1D} Graph of the random function $F_{\lambda,\alpha,\beta}$ when $D=1$.}
\end{figure}

\begin{figure}[!h]
\centering
\subfigure[$(\lambda,\alpha,\beta)=(1.5,1,1)$]{\includegraphics[width=7cm]{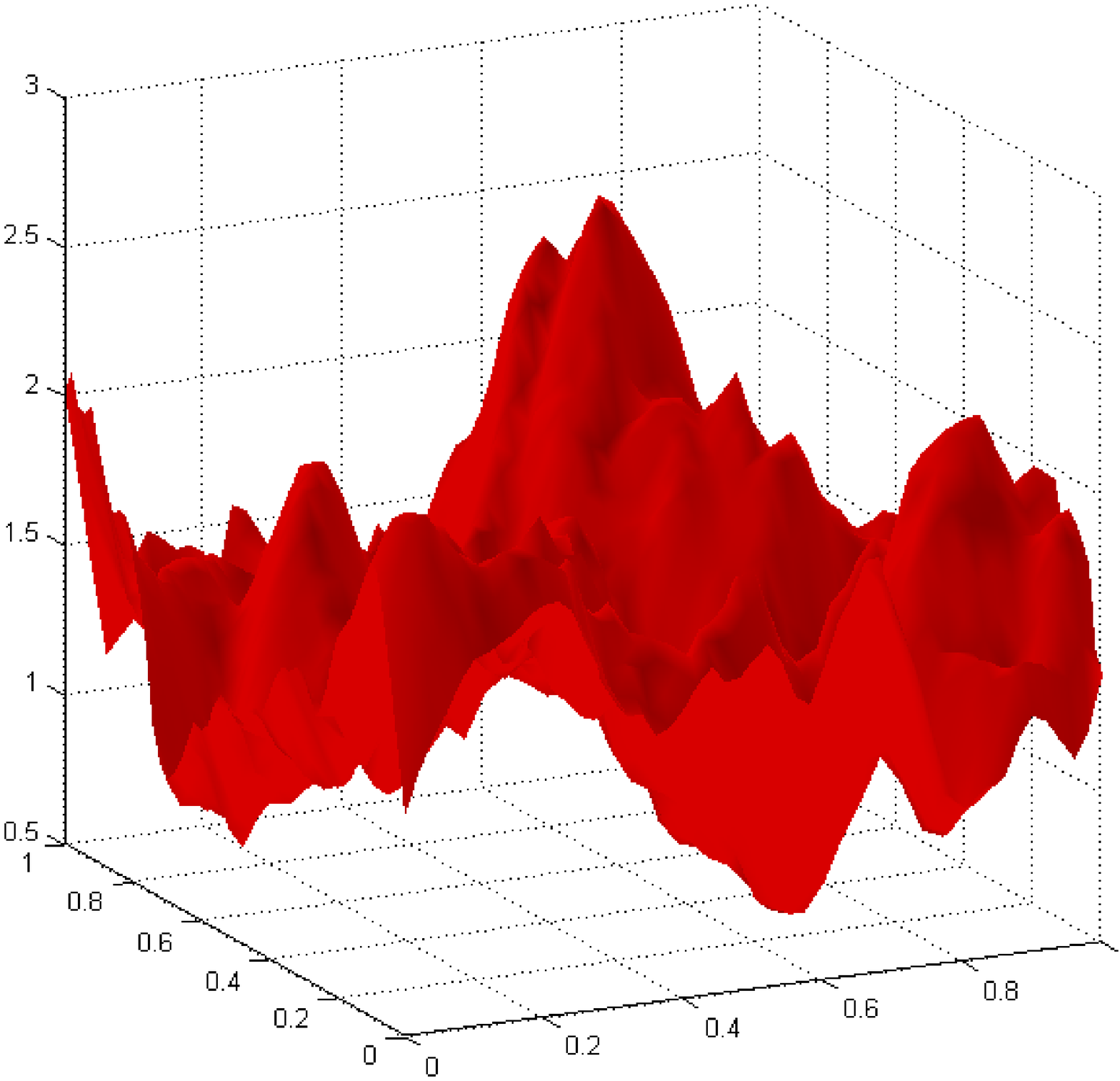}}\goodgap
\subfigure[$(\lambda,\alpha,\beta)=(1.5,0.2,1)$]{\includegraphics[width=7cm]{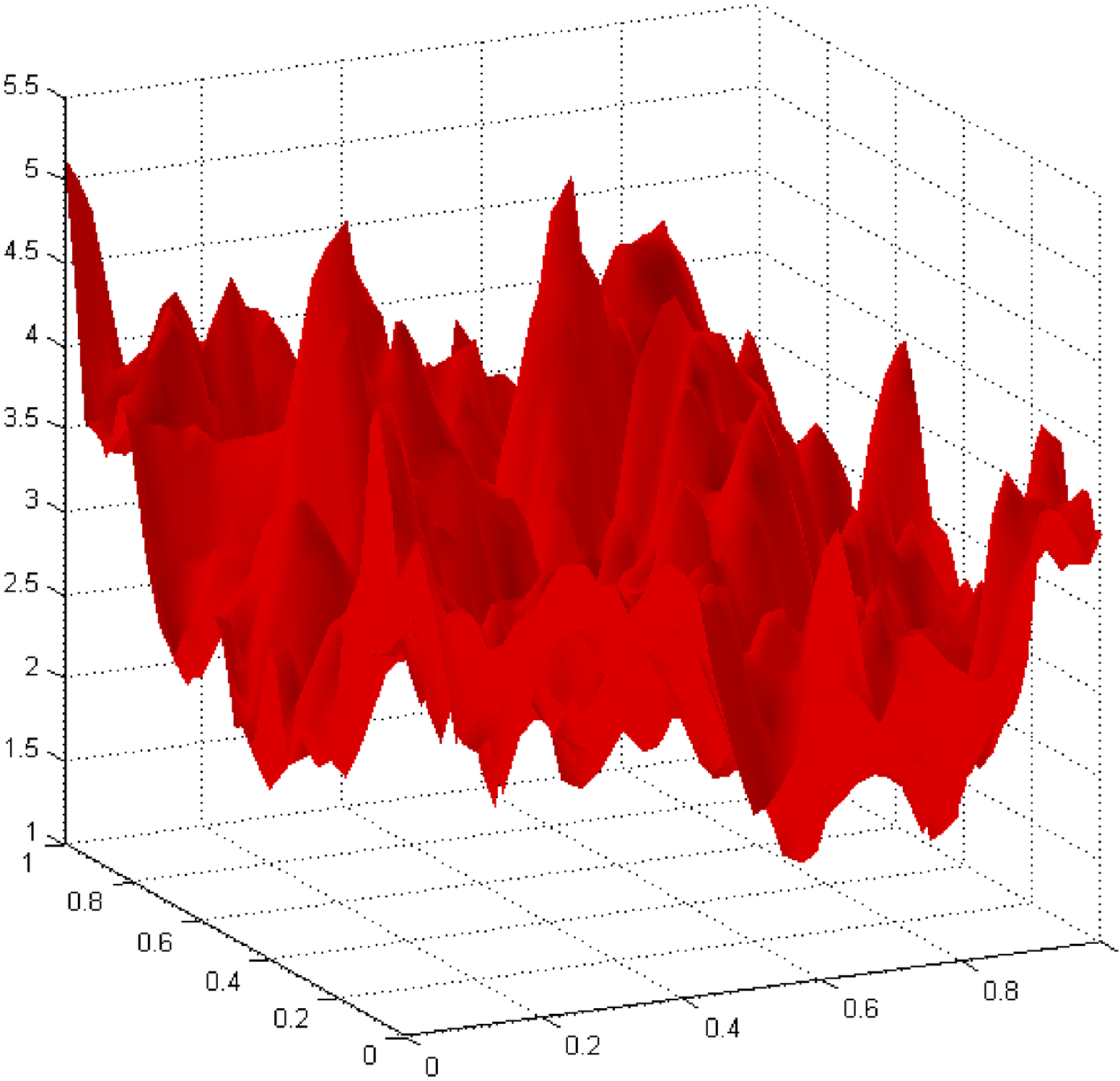}}\goodgap \\
\vspace{-0.2cm}
\caption{\label{fig:modelpoiss2D} Graph of the random function $F_{\lambda,\alpha,\beta}$ when $D=2$.}
\end{figure}

\pass The paper is organized as follows. In the first section we state some preliminary results related to the geometry of the Poisson-Voronoi tessellations. We introduce in particular the oscillation sets ${\mathcal O}_{n,N}$ (see \eqref{eq:randomset}) that are used to derive explicit distributional properties on the increments of $\Delta_n$ and precise estimates on the increments of $F_{\lambda,\alpha,\beta}$. Section 2 is then devoted to the proof of Theorem \ref{theo:maintheo}. An upper bound for $\dimbox(\Gamma_{\lambda,\alpha,\beta})$ comes from the estimation of the oscillations of $F_{\lambda,\alpha,\beta}$ whereas a lower bound for $\dimhaus(\Gamma_{\lambda,\alpha,\beta})$ is obtained via a Frostman-type lemma. Finally, we introduce and study in the last section two related models: a deterministic series based on an hexagonal mesh and a random series based on a perturbation of the dyadic mesh.

\pass In the sequel we will drop the indices $\lambda$, $\alpha$ and $\beta$ so that $F=F_{\lambda,\alpha,\beta}$ and $\Gamma=\Gamma_{\lambda,\alpha,\beta}$.

\pass
\section{Preliminary results}\label{sec:prelim}

\subsection{Notations}\label{subsec:notations}

\pass

\noindent We consider the metric space $\R^D$, $D\geqslant1$, endowed with the Euclidean norm $\|\cdot\|$. The closed ball with center $x\in\R^D$ and radius $r>0$ is denoted by $B_r(x)$. We write $\vol(A)$ for the Lebesgue measure of a Borel set $A\subset\R^D$. In particular $\kappa_D=\vol(B_1(0))$. The unit sphere of $\R^D$ is denoted by ${\mathbb S}^{D-1}$ and $\sigma_{D-1}$ will be the unnormalized area measure on ${\mathbb S}^{D-1}$. The surface area of ${\mathbb S}^{D-1}$ is then $\omega_{D-1}=\sigma_{D-1}({\mathbb S}^{D-1})$. Finally, for all $s\geqslant0$, the $s$-dimensional Hausdorff measure is ${\mathcal H}^s$.

\pass For all $x,y\in[0,1]^D$ and all $n\geqslant 0$ let
\begin{equation}\label{eq:defZn}
Z_n(x,y) = \lambda^{-\frac{n\alpha}{D}}(\Delta_n(x)-\Delta_n(y))
\end{equation}
so that $F(x)-F(y)= \sum_{n=0}^{\infty} Z_n(x,y)$, and
\begin{equation}\label{eq:Sn}
S_n(x,y) = \sum_{\substack{m=0 \\ m\neq n}}^{\infty} Z_m(x,y)
\end{equation}
so that $F(x)-F(y)=Z_n(x,y)+S_n(x,y)$. Notice that $Z_n(x,y)$ and $Z_m(x,y)$ are two independent random variables for $m\neq n$. In particular $Z_n(x,y)$ and $S_n(x,y)$ are independent.

\pass Finally, we fix $H>\beta$. For all $n\geqslant0$, we set $\tau_n=\lambda^{-\frac{nH}{D}}$.

\subsection{Random oscillation sets}\label{subsec:randomsets}

\pass

\noindent Remember that the function $\Delta_n$ is piecewise linear. Any maximal set on which $\Delta_n$ is linear is the convex hull $\mbox{Conv}(\{c\}\cup f)$ of the union of a nucleus $c$ from ${\mathcal X}_n$ and a hyperface (i.e. a $(D-1)$-dimensional face) $f$ of the cell associated with $c$. The set ${\mathcal S}_n$ of such simplices tessellates $\R^D$. For all $n,N\geqslant 0$ we define the random sets
\begin{equation}\label{eq:randomset}
{\mathcal O}_{n,N}= \left\{x\in [0,1]^D : \mbox{all points of $B_{\tau_N}(x)$ are in the same simplex of ${\mathcal S}_n$ as $x$} \right\}
\end{equation}
and
\begin{equation}\label{eq:WN}
W_N=\bigcap_{n=N}^{\infty} {\mathcal O}_{n,n}.
\end{equation}
The set ${\mathcal O}_{n,N}$ is referenced as random oscillation set because the oscillations of the function $\Delta_n$ can be properly estimated only on such set.

\pass The first result states that these sets are not too `small'.
\begin{prop}\label{prop:size} \hfill
\begin{enumerate}
\vspace{-0.1cm}
\item[(i)] There exists a constant $C>0$ such that, for all $x\in[0,1]^D$ and all $N\geqslant n\geqslant 0$,
\begin{equation*}\label{eq:probaon}
\P(x\not\in {\mathcal O}_{n,N})\leqslant C \lambda^{\frac{n\beta-NH}{D}}.
\end{equation*}
\vspace{-0.3cm}
\item[(ii)] We have $\,\lim_{N\to\infty}\P(\vol(W_{N})>0)=1.$
\end{enumerate}
\end{prop}

\begin{proof} \hfill

\noindent (i) By invariance by translation of ${\mathcal X}_n$ and ${\mathcal T}_n$, we notice that for every $x\in\R^D$,
\begin{equation}\label{eq:stationarity}
\P(0\not\in {\mathcal O}_{n,N})=\P(x\not\in {\mathcal O}_{n,N}).
\end{equation}

\noindent Let $\sk_n$ be the skeleton of the simplex tessellation ${\mathcal S}_n$, i.e. the union of the boundaries of all simplices
(see the grey region on Figure \ref{fig:skeleton}).

\begin{figure}[!h]
    \centering
    \includegraphics[width=11cm]{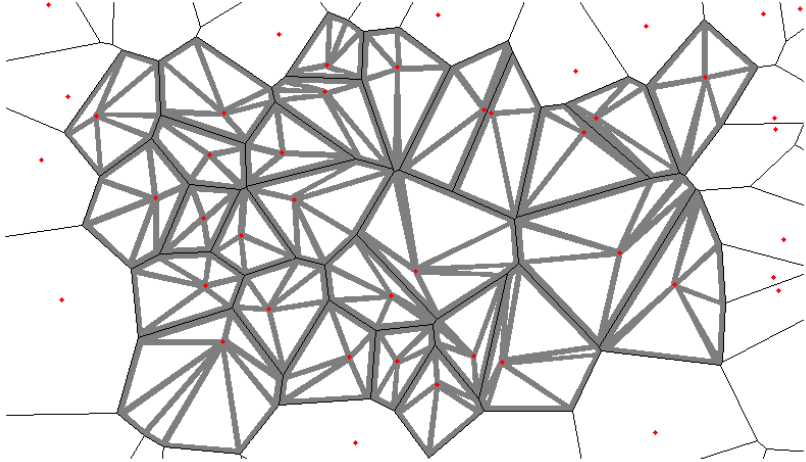}
    \caption{\label{fig:skeleton} The skeleton of the complete tessellation (in grey).}
\end{figure}

\noindent In particular, we have the equivalence
\begin{equation}\label{eq:skeleton}
x\not\in {\mathcal O}_{n,N}\Longleftrightarrow x\in \sk_n+B_{\tau_N}(0).
\end{equation}

\noindent Let $U$ be a uniform point in $[0,1]^D$, independent of the tessellation ${\mathcal T}_n$, and $\P_{{\mathcal X}_n}$ be the distribution of the Poisson point process ${\mathcal X}_n$. Using (\ref{eq:stationarity}) and Fubini's theorem, we get
\begin{equation*}
\P(U\not\in {\mathcal O}_{n,N})=\int_{[0,1]^D}\P_{{\mathcal X}_n}(x\not\in {\mathcal O}_{n,N})\mathrm{d}x=\P(0\not\in  {\mathcal O}_{n,N}).
\end{equation*}

\noindent Moreover, the equivalence (\ref{eq:skeleton}) implies that
\begin{equation}\label{eq:0notin}
\P(0\not\in {\mathcal O}_{n,N}) = \E_{{\mathcal X}_n}\biggl(\int_{[0,1]^D}\indic_{\sk_n+B_{\tau_N}(0)}(x)\dd x\biggl)
= \E_{\cal{X}_n}\bigl(\vol((\sk_n+B_{\tau_N}(0))\cap [0,1]^D)\bigl).
\end{equation}

\noindent It remains to calculate the Lebesgue measure of the set $(\sk_n+B_{\tau_N}(0))\cap [0,1]^D$. Denoting by ${\mathcal F}_{n}(a)$ the set of hyperfaces of the simplex tessellation ${\mathcal S}_n$ which intersect $[0,a]^D$ we have
\begin{equation*}
\vol((\sk_n+B_{\tau_N}(0))\cap [0,1]^D) \leqslant  \sum_{f\in {\mathcal F}_{n}(1)} \vol(f+B_{\tau_N}(0))+\vol(\partial([0,1]^D)+B_{\tau_N}(0)).
\end{equation*}
By Steiner formula (see e.g. Prolog in \cite{SW}), we have for every $f\in {\mathcal F}_n(1)$,
\begin{equation*}
\vol(f+B_{\tau_N}(0))=\sum_{i=0}^{D-1}\lambda^{-\frac{NH(D-i)}{D}}\kappa_{D-i}V_i(f)
\end{equation*}
where $V_i(f)$ is the $i$-th intrinsic volume of $f$.

\noindent Consequently, we have
\begin{align}\label{eq:airesquelette}
\E_{{\mathcal X}_n}\big(\vol((\sk_n+B_{\tau_N}(0)) & \cap [0,1]^D)\big) \nonumber\\
& \leqslant \sum_{i=0}^{D-1}\lambda^{-\frac{NH(D-i)}{D}}\kappa_{D-i}\E_{{\mathcal X}_n}\bigg(\sum_{f\in {\mathcal F}_{n}(1)} V_{i}(f)\bigg)+4D\lambda^{-\frac{NH}{D}}.
\end{align}

\noindent Using the invariance of ${\mathcal X}_n$ by scaling transformations and translations and the fact that $V_i$ is a homogeneous function of degree $i$, we observe that for every $0\leqslant i\leqslant D-1$,
\begin{align}\label{eq:scaling}
\E_{{\mathcal X}_n}\bigg(\sum_{f\in {\mathcal F}_{n}(1)}V_{i}(f)\bigg) = \lambda^{-\frac{n\beta i}{D}}
\,\E\bigg(\sum_{f\in {\mathcal F}_{0}(\lambda^{\frac{n\beta}{D}})}V_{i}(f) \bigg)
= \lambda^{-\frac{n\beta i}{D}}\,\E\bigg(\sum_{f\in {\mathcal F}_{0}(1)}V_{i}(f)\bigg)(\lambda^{\frac{n\beta}{D}})^D.
\end{align}

\noindent Combining (\ref{eq:0notin}), (\ref{eq:airesquelette}) and (\ref{eq:scaling}), we obtain the required result (i).

\noindent (ii) The point (i) implies that
\begin{equation*}
\E(\vol([0,1]^D\setminus {\mathcal O}_{n,n})) = \int_{[0,1]^D} \P(x\notin {\mathcal O}_{n,n}) \dd x = \P(0\notin {\mathcal O}_{n,n}) = \mathrm{O}(\lambda^{\frac{n(\beta-H)}{D}}).
\end{equation*}
Therefore, we obtain
\begin{equation*}
\E(\vol([0,1]^D\setminus W_{N}))\leqslant \sum_{n=N}^{\infty}\E(\vol([0,1]^D\setminus {\mathcal O}_{n,n}))\leqslant\sum_{n=N}^{\infty}\mathrm{O}(\lambda^{\frac{n(\beta-H)}{D}}) = \mathrm{O}(\lambda^{\frac{N(\beta-H)}{D}}).
\end{equation*}
Finally, using Markov's inequality,
\begin{equation*}
\P(\vol(W_{N})<1/2)=\P(\vol([0,1]^D\setminus W_{N})\geqslant 1/2) \leqslant 2\E(\vol([0,1]^D\setminus W_{N}))\leqslant \mathrm{O}(\lambda^{\frac{N(\beta-H)}{D}})
\end{equation*}
and the result (ii) follows.
\end{proof}

\subsection{Distribution of the random variable $Z_n(x,y)$}\label{subsec:distZn}

\pass

\noindent This subsection is devoted to the calculation of the distribution of $Z_n(x,y)$ conditionally on $\{x\in {\mathcal O}_{n,n}\}$ when $\|x-y\|\leqslant \tau_n$. In particular, we obtain in Proposition \ref{prop:denshn} below an explicit formula and an upper-bound for the conditional density of $Z_n(x,y)$. A similar method provides in Proposition \ref{prop:lipconst} the integrability of the local Lipschitz constant of $Z_0$. All these results will play a major role in the estimation of the oscillations of $F$ (see Proposition \ref{prop:estimoscill}) and the application of the Frostman criterion (see Proposition \ref{prop:majespercond}).

\pass For any $x\in [0,1]^D$, let $c_n(x)$ (resp. ${\mathcal C}_n(x)$) be the nucleus (resp. the cell) from the Voronoi tessellation ${\mathcal T}_n$ associated with $x$, i.e. the point of ${\chi_n}$ which is the closest to $x$ (resp. the cell of such point). Let $c'_n(x)$ be the `secondary nucleus' of $x$, i.e. the point of $\cal{X}_n$ which is the nucleus of the neighboring cell of ${\mathcal C}_n(x)$ in the direction of the half-line $[c_n(x),x)$. Moreover, for any $z_1\neq z_2\in\R^D$ and $x\in \R^D\setminus (B_{\tau_n}(z_1)\cup B_{\tau_n}(z_2))$, we consider
\begin{enumerate}
\item[-] the bisecting hyperplane $H_{z_1,z_2}$ of $[z_1,z_2]$,
\item[-] the cone $\Lambda(z_1,x)$ of apex $z_1$ and generated by the ball $B_{\tau_n}(x)$,
\item[-] the set $A_{n,x}$ of couples $(z_1,z_2)$ with $z_1\not\in B_{\tau_n}(x)$ such that $x$ is between the hyperplane orthogonal to $z_2-z_1$ and containing $z_1$ and the parallel hyperplane which is at distance $\tau_n$ from $H_{z_1,z_2}$ on the $z_1$-side:
      \begin{equation*}
            A_{n,x}=\biggl\{(z_1,z_2)\in B_{\tau_n}(x)^c\times\R^D :\,0\leqslant \bigl\langle x-z_1,z_2-z_1\bigl\rangle\leqslant\frac{1}{2}\,\|z_2-z_1\|^2\bigg(1-\frac{2\tau_n}{\|z_2-z_1\|}\bigg)\biggl\}.
      \end{equation*}
\end{enumerate}
Finally, we denote by $\cal{V}_n(x,z_1,z_2)$ the volume of the Voronoi flower associated with the intersection $H_{z_1,z_2}\cap \Lambda(z_1,x)$:
\begin{equation*}
\cal{V}_n(x,z_1,z_2) =\vol\left(\,\bigcup\,\left\{B_{\|u-z_1\|}(u):\,u\in H_{z_1,z_2}\cap \Lambda(z_1,x)\right\}\right).
\end{equation*}

\vspace{0.15cm}
\begin{figure}[!h]
    \centering
    \includegraphics[width=10cm]{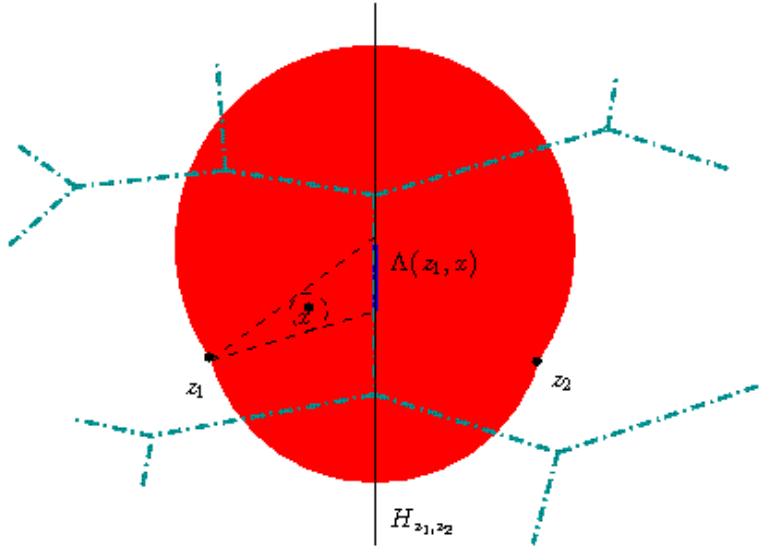}
    \caption{\label{fig:flower} The configuration of $(z_1,z_2)$ with the associated Voronoi flower (in red).}
\end{figure}

\newpage
\begin{prop}\label{prop:denshn}
Let $n\geqslant 0$ and $x,y\in[0,1]^D$ such that $x\in {\mathcal O}_{n,n}$ and $0<\|x-y\|\leqslant \tau_n$. Then
\begin{enumerate}
\item[(i)] The increment $Z_n(x,y)$ is given by
\begin{equation}\label{diffhauteur}
Z_n(x,y)=-\frac{2\lambda^{-\frac{n\alpha}{D}}}{\|c_n'(x)-c_n(x)\|^2}\,\bigl\langle x-y,c_n'(x)-c_n(x)\bigl\rangle.
\end{equation}
\item[(ii)] The density $g_{Z_n}$ of $Z_n(x,y)$ conditionally on $\{x\in {\mathcal O}_{n,n}\}$ is given by \eqref{densitezn} for $D\geqslant 2$ and by \eqref{eq:densitezncasD=1} for $D=1$. Moreover, it satisfies
\begin{equation}\label{eq:norminfdenshn}
\sup_{t\in\R}\,g_{Z_n}(t)  \leqslant \frac{C}{\P(x\in {\mathcal O}_{n,n})} \|x-y\|^{-1} \,\lambda^{-\frac{n(\beta-\alpha)}{D}}
\end{equation}
where $C$ is a positive constant depending only on the dimension $D$.
\end{enumerate}
\end{prop}
\begin{proof} \hfill

\noindent (i) If $x\in {\mathcal O}_{n,n}$ and $\|x-y\|\leqslant \tau_n$ then
\begin{equation*}
\Delta_n(x)=\frac{\dist(x,H_{c_n(x),c_n'(x)})}{\dist(c_n(x),H_{c_n(x),c_n'(x)})}.
\end{equation*}
Moreover,
\begin{equation*}
\dist(x,H_{c_n(x),c_n'(x)})=\biggl\langle x -\frac{c_n(x)+c_n'(x)}{2}, \frac{c_n(x)-c_n'(x)}{\|c_n(x)-c_n'(x)\|}\biggl\rangle
\end{equation*}
and
\begin{equation*}
\dist(c_n(x),H_{c_n(x),c_n'(x)})=\frac{1}{2}\,\|c_n(x)-c_n'(x)\|.
\end{equation*}
It remains to use the definition (\ref{eq:defZn}) of $Z_n(x,y)$ to obtain the result (i).

\vspace{0.15cm}
\begin{figure}[!h]
    \centering
    \includegraphics[width=10cm]{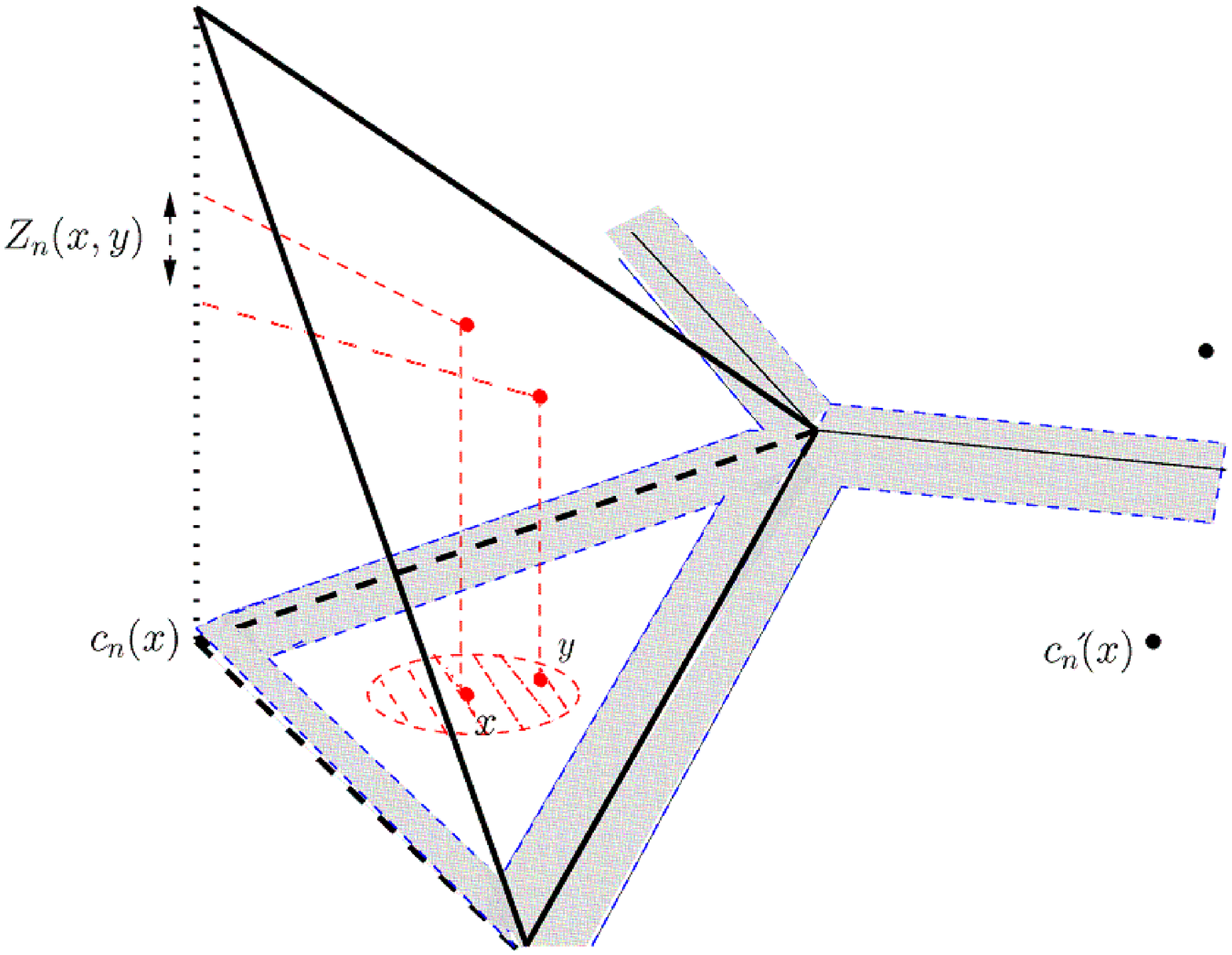}
    \caption{\label{fig:hauteurzn} The random variable $Z_n(x,y)$.}
\end{figure}

\noindent (ii) We need to determine the joint distribution of $(c_n(x),c_n'(x))$. We notice that $x$ belongs to ${\mathcal O}_{n,n}$
if and only if all the points of $B_{\tau_n}(x)$ have the same nucleus as $x$ in ${\mathcal X}_n$ and same `secondary nucleus'. In other words,
$$x\in {\mathcal O}_{n,n}\Longleftrightarrow \Lambda(c_n(x),x)\cap H_{c_n(x),c_n'(x)}\subset {\mathcal C}_{c_n(x)}\cap {\mathcal C}_{c_n'(x)}.$$
\noindent By Mecke-Slivnyak's formula (see Corollary $3.2.3$ in \cite{SW}), for any measurable and non-negative function $h:\R^2\longrightarrow \R_+$ we have
\begin{align*}
& \hspace*{-1cm} \E(h(c_n(x),c_n'(x))\indic_{\{x\in{\mathcal O}_{n,n}\}})\\& = \E\biggl(\,\sum_{z_1\neq z_2}h(z_1,z_2)\indic_{(c_n(x),c_n'(x))}(z_1,z_2)\indic_{\{x\in{\mathcal O}_{n,n}\}}\biggl)\\
& = \lambda^{2n\beta}\iint h(z_1,z_2)\P(\Lambda(z_1,x)\cap H_{z_1,z_2}\subset {\mathcal C}_{z_1}\cap{\mathcal C}_{z_2})\indic_{A_{n,x}}(z_1,z_2)\dd z_1 \dd z_2\\
& =\lambda^{2n\beta}\iint h(z_1,z_2)\exp\bigl(-\lambda^{n\beta}\cal{V}_n(x,z_1,z_2)\bigl)\indic_{A_{n,x}}(z_1,z_2)\dd z_1 \dd z_2.
\end{align*}

\noindent We point out a small abuse of notation above: the sets ${\mathcal C}_{z_1}$ and ${\mathcal C}_{z_2}$ are Voronoi cells associated with $z_1$ and $z_2$ when the underlying set of nuclei is ${\mathcal X}_{n}\cup \{z_1,z_2\}$. We proceed now with the change of variables $z_2=(z_2)_{\rho,u}=z_1+\rho u$ with $\rho>0$ and $u\in{\mathbb S}^{D-1}$.

\noindent Let $L_n(x)=\|c_n'(x)-c_n(x)\|$, $u_n(x)=\frac{c_n'(x)-c_n(x)}{L_n(x)}$ and ${\mathcal V}_n'(\cdot)={\mathcal V}_n(x,z_1,z_1+\cdot)$
(keep in mind that ${\mathcal V}_n'$ will still depend on $x$ and $z_1$ though this dependency will not be visible for sake of readability).
The density of $(L_n(x),u_n(x))$ conditionally on $\{x\in {\mathcal O}_{n,n}\}$ with respect to the product measure $\mathrm{d}\rho \mathrm{d}\sigma_{D-1}(u)$ is
\begin{equation}\label{densityLn}
\frac{\lambda^{2n\beta}}{\P(x\in {\mathcal O}_{n,n})}\int_{B_{\tau_n}(x)^c} \exp\bigl(-\lambda^{n\beta}\cal{V}_n'(\rho u)\bigl)\indic_{[0,\frac{\rho}{2}-\tau_n]}\bigl(\bigl\langle x-z_1,u\bigl\rangle\bigl)\rho^{D-1}\mathrm{d}z_1.
\end{equation}

\noindent Using \eqref{diffhauteur}, we can rewrite the quantity $Z_n(x,y)$ as a function of $L_n(x)$ and $u_n(x)$ as
\begin{equation*}
Z_n(x,y)=-\frac{2\lambda^{-\frac{n\alpha}{D}}}{L_n(x)}\bigl\langle x-y,u_n(x)\bigl\rangle.
\end{equation*}

\noindent The density of the distribution of $Z_n(x,y)$ conditionally on $\{x\in {\mathcal O}_{n,n}\}$ can then be calculated
in the following way: for any non-negative measurable function $\psi:\R_+\longrightarrow\R_+$,
\begin{equation}\label{densiteinterm}
\E(\psi(Z_n(x,y))) = \frac{\lambda^{2n\beta}}{\P(x\in {\mathcal O}_{n,n})}\int_{B_{\tau_n}(x)^c} J_n(x,y,z_1,\rho,u)\mathrm{d}z_1
\end{equation}
where $J_n(x,y,z_1,\rho,u)$ is equal to
\begin{align*}
\iint\psi\bigl(-2\lambda^{-\frac{n\alpha}{D}} \rho^{-1}\langle x-y,u\rangle\bigl)e^{-\lambda^{n\beta}\cal{V}_n'(\rho u)}\indic_{\R_+}(\langle x-z_1,u\rangle)\rho^{D-1}\mathrm{d}\rho\,\mathrm{d}\sigma_{D-1}(u),
\end{align*}
the domain of integration for $\rho$ being $[2\tau_n+2\langle x-z_1,u\rangle,\infty)$.
~\\~\\
\noindent {\it Case $D=1$}. We observe that $u =\pm 1$ and $\langle x-z_1,u\rangle=(x-z_1)u=\pm (x-z_1)$. Moreover, the condition $z_1\not\in B_{\tau_n}(x)$ means that $z_1<x-\tau_n$ or $z_1>x+\tau_n$. It implies that if $u=1$ (resp. $u=-1$), the range of $z_1$ is $(-\infty,x-\tau_n)$ (resp. $(x+\tau_n,\infty)$) while for fixed $z_1$ the range of $\rho$ is $(2\tau_n+2(x-z_1),\infty)$ (resp. $(2\tau_n+2(z_1-x),\infty)$). Finally, the set $\bigcup\,\bigl\{B_{\|u-z_1\|}(u):\,u\in H_{z_1,z_2}\cap \Lambda(z_1,x)\bigl\}$ is $[z_1,z_1+\rho]$ (if $u=1$) or $[z_1-\rho,z_1]$ (if $u=-1$) so ${\mathcal V}_n'(\rho u)=\rho$ in both cases. Consequently, we have
\begin{equation}\label{cased=1decomp}
\E(\psi(Z_n(x,y))) = \frac{\lambda^{2n\beta}}{\P(x\in {\mathcal O}_{n,n})}(I_1+I_2)
\end{equation}
where
$$ I_1=\int_{-\infty}^{x-\tau_n}\bigg(\int_{2\tau_n+2(x-z_1)}^{+\infty}\psi(-2\lambda^{-n\alpha}\rho^{-1}(x-y))e^{-\lambda^{n\beta}\rho}
\mathrm{d}\rho\bigg)\mathrm{d}z_1$$
and
$$ I_2=\int_{x+\tau_n}^{+\infty}\bigg(\int_{2\tau_n+2(z_1-x)}^{+\infty}\psi(2\lambda^{-n\alpha}\rho^{-1}(x-y))e^{-\lambda^{n\beta}\rho}
\mathrm{d}\rho\bigg)\mathrm{d}z_1.$$
Applying Fubini's theorem in $I_1$ then the change of variables $\rho'=-\rho$, we get
\begin{align}\label{cased=1I1}
I_1&=\int_{4\tau_n}^{+\infty}\psi(-2\lambda^{-n\alpha}\rho^{-1}(x-y))e^{-\lambda^{n\beta}\rho}
\bigg(\int_{x-\frac{\rho}{2}+\tau_n}^{x-\tau_n}\mathrm{d}z_1\bigg)\mathrm{d}\rho\nonumber\\
&=\int_{4\tau_n}^{+\infty}\psi(-2\lambda^{-n\alpha}\rho^{-1}(x-y))e^{-\lambda^{n\beta}\rho}\bigg(\frac{\rho}{2}-2\tau_n\bigg)\mathrm{d}\rho\nonumber\\
&=\int_{-\infty}^{-4\tau_n}\psi(2\lambda^{-n\alpha}\rho'^{-1}(x-y))e^{-\lambda^{n\beta}|\rho'|}\bigg(\frac{|\rho'|}{2}-2\tau_n\bigg)\mathrm{d}\rho'.
\end{align}
Applying Fubini's theorem in $I_2$, we obtain
\begin{align}\label{cased=1I2}
I_2&=\int_{4\tau_n}^{+\infty}\psi(2\lambda^{-n\alpha}\rho^{-1}(x-y))e^{-\lambda^{n\beta}\rho}
\bigg(\int_{x+\tau_n}^{x+\frac{\rho}{2}-\tau_n}\mathrm{d}z_1\bigg)\mathrm{d}\rho\nonumber\\&=\int_{4\tau_n}^{+\infty}
\psi(2\lambda^{-n\alpha}\rho^{-1}(x-y))e^{-\lambda^{n\beta}\rho}\bigg(\frac{\rho}{2}-2\tau_n\bigg)\mathrm{d}\rho.
\end{align}
Combining \eqref{cased=1decomp}, \eqref{cased=1I1} and \eqref{cased=1I2}, we get
$$\E(\psi(Z_n(x,y))) = \frac{\lambda^{2n\beta}}{\P(x\in {\mathcal O}_{n,n})}
\int_{|\rho|>4\tau_n}\psi(2\lambda^{-n\alpha}\rho^{-1}(x-y))e^{-\lambda^{n\beta}|\rho|}\bigg(\frac{|\rho|}{2}-2\tau_n\bigg)\mathrm{d}\rho.$$
Applying the change of variables $\rho=\rho_{t}=2\lambda^{-n\alpha}t^{-1}(x-y)$, we get that the density of $Z_n(x,y)$ is, for $|t|<\lambda^{-n\alpha}\frac{|x-y|}{2\tau_n}$,
\begin{equation}\label{eq:densitezncasD=1}
g_{Z_n}(t)=\frac{2\lambda^{2n\beta}}{\P(x\in {\mathcal O}_{n,n})}e^{-2\lambda^{n(\beta-\alpha)}\frac{|x-y|}{|t|}}
\left(\lambda^{-n\alpha}\frac{|x-y|}{|t|}-2\tau_n\right)\lambda^{-n\alpha}\frac{|x-y|}{t^2}.
\end{equation}
In particular,
$$g_{Z_n}(t)\leqslant\frac{2\lambda^{2n\beta}}{\P(x\in {\mathcal O}_{n,n})}\frac{\lambda^{-3n\beta+n\alpha}}{|x-y|}\sup_{r>0}(e^{-2r}r^3),$$
which shows \eqref{eq:norminfdenshn}.
~\\~\\
\noindent{\it Case $D\geqslant 2$}. We go back to \eqref{densiteinterm}.
For almost any $u\in\mathbb{S}^{D-1}$, there exists a unique $v\in \mathbb{S}^{D-1}\cap \{y-x\}^{\perp}$ and a unique $s=\bigl\langle u,\frac{y-x}{\|y-x\|}\bigl\rangle\in(-1,1)$ such that
\begin{equation*}
u=u_{s,v}=s\frac{y-x}{\|y-x\|}+\sqrt{1-s^2}\,v.
\end{equation*}
In particular, we can rewrite the uniform measure of ${\mathbb S}^{D-1}$ as
\begin{equation*}
\mathrm{d}\sigma_{D-1}(u)=(1-s^2)^{\frac{D-3}{2}}\mathrm{d}s\, d\sigma_{D-2}(v).
\end{equation*}

\noindent We thus get that $J_n(x,y,z_1,\rho,u)$ is also equal to
\begin{align*}
\iiint\psi\bigl(2\lambda^{-\frac{n\alpha }{D}}\|x-y\|s\rho^{-1}\bigl)e^{-\lambda^{n\beta}\cal{V}_n'(\rho u_{s,v})}\indic_{\R_+}(\langle x-y,u_{s,v}\rangle)\rho^{D-1}(1-s^2)^{\frac{D-3}{2}}\mathrm{d}\rho\,\mathrm{d}s\, d\sigma_{D-2}(v),
\end{align*}
the domain of integration for $\rho$ being $[2\tau_n+2\langle x-z_1,u_{s,v}\rangle,\infty)$.

\noindent We now proceed with the change of variables $\rho=\rho_t=2\lambda^{-\frac{n\alpha}{D}}\|x-y\|st^{-1}$ with $st>0$. We then deduce
that the density $g_{Z_n}(t)$ at point $t$ of $Z_n(x,y)$ conditionally on $\{x\in {\mathcal O}_{n,n}\}$ is given by
\begin{equation}\label{densitezn}
g_{Z_n}(t)=
\frac{\lambda^{2n\beta}}{\P(x\in {\mathcal O}_{n,n})}\iiint J_n'(x,y,z_1,t,s,v)\indic_{D_n}(x,y,z_1,t,s,v)\mathrm{d}s\mathrm{d}\sigma_{D-2}(v)\mathrm{d}z_1
\end{equation}
where
\begin{equation*}
J_n'(x,y,z_1,t,s,v)=e^{-\lambda^{n\beta}\cal{V}_n'\big(\frac{2\lambda^{-\frac{n\alpha}{D}}\|x-y\|\,s\,u_{s,v}}{t}\big)}
\bigg(\frac{2\lambda^{-\frac{n\alpha}{D}}\|x-y\|s}{t}\bigg)^D \,\frac{(1-s^2)^{\frac{D-3}{2}}}{t}
\end{equation*}
and
\begin{equation*}
D_n=\bigg\{(x,y,z_1,t,s,v): \|x-z_1\|>\tau_n \mbox{ and } 0\leqslant \langle x-z_1,u_{s,v}\rangle \leqslant \frac{\lambda^{-\frac{n\alpha}{D}}\|x-y\|s}{t}-\tau_n\bigg\}.
\end{equation*}

\noindent In the sequel, we only deal with the case $t>0$ but the same could be done likewise for $t<0$. We denote by $x'$ the intersection of the half-line $[z_1,x)$ with the boundary of the Voronoi cell of $z_1$.
Moreover, we write $z_1=x-\gamma w$ where $\gamma>\tau_n$ and $w\in{\mathbb S}^{D-1}$. In particular, we notice that
\begin{equation*}
\|x'-z_1\|=\frac{\rho}{2\langle w,u_{s,v}\rangle}=\frac{\lambda^{-\frac{n\alpha}{D}}\|x-y\|s}{t\langle w,u_{s,v}\rangle}.
\end{equation*}
We can now easily estimate the volume ${\mathcal V}_n'(\cdot)$ in the following way:
\begin{equation}\label{eq:minorvolume}
{\mathcal V}_n'(\cdot)\geqslant \vol(B_{\|x'-z_1\|}(x')) =\kappa_D\biggl(\frac{\lambda^{-\frac{n\alpha}{D}}\|x-y\|s}{t\langle w,u_{s,v}\rangle}\biggl)^D.
\end{equation}

\noindent We then proceed with the following change of variables: for almost any $w\in {\mathbb S}^{D-1}$, there exist a unique $\xi=\langle w,u_{s,v}\rangle \in[0,1)$ and a unique $\eta\in\mathbb{S}^{D-1}\cap \{u_{s,v}\}^{\perp}$ such that $w=w_{\xi,\eta}=\xi u_{s,v}+ \sqrt{1-\xi^2}\,\eta$ and
\begin{equation}\label{eq:changeofvariablemarcheD=2aussi}
\mathrm{d}z_1=\gamma^{D-1}\mathrm{d}\gamma\,\mathrm{d}\sigma_{D-1}(w)
=\gamma^{D-1}(1-\xi^2)^{\frac{D-3}{2}}\mathrm{d}\gamma\,\mathrm{d}\xi\,\mathrm{d}\sigma_{D-2}(\eta).
\end{equation}

\noindent In particular, when $(x,y,z_1,t,s,v)\in D_n$, we have
\begin{equation}\label{eq:intgamma}
0\leqslant\gamma=\frac{\langle x-z_1,u_{s,v}\rangle}{\xi}\leqslant \frac{\lambda^{-\frac{n\alpha}{D}}\|x-y\|s}{t\xi}.
\end{equation}
Consequently, for fixed $s,\xi \in (0,1)$, we have
\begin{equation}\label{eq:integraleengamma}
\iiint \indic_{D_n}\gamma^{D-1}\mathrm{d}\gamma\mathrm{d}\sigma_{D-2}(v)\mathrm{d}\sigma_{D-2}(\eta)\leqslant \frac{\omega_{D-2}^2}{D}
\left(\frac{\lambda^{-\frac{n\alpha}{D}}\|x-y\|s}{t\xi}\right)^D.
\end{equation}
We deduce from \eqref{densitezn}, \eqref{eq:minorvolume}, \eqref{eq:changeofvariablemarcheD=2aussi} and \eqref{eq:integraleengamma} that the density $g_{Z_n}(t)$ satisfies, for every $t>0$,
\begin{equation}\label{eq:dernierelignemarchepourtoutD}
g_{Z_n}(t) \leqslant \frac{\lambda^{2n\beta}\omega_{D-2}^2}{D\P(x\in {\mathcal O}_{n,n})}\int_0^1 \int_0^1 J''_n(x,y,t,s,\xi)\,\mathrm{d}s\,\mathrm{d}\xi
\end{equation}
where
\begin{equation*}\label{eq:dernierelignemarchepourtoutDbis}
J''_n(x,y,t,s,\xi)= e^{-\lambda^{n\beta}\kappa_D\bigl(\frac{\lambda^{-\frac{n\alpha}{D}}\|x-y\|s}{t\xi}\bigl)^D}
\frac{(\sqrt{2}\lambda^{-\frac{n\alpha}{D}}\|x-y\|s)^{2D}}{t^{2D+1}\xi^D}\,\left((1-s^2)(1-\xi^2)\right)^{\frac{D-3}{2}}.
\end{equation*}

\noindent{\it Subcase $D\geqslant 3$}. We then use the change of variables $s=s_{\tau}=\lambda^{\frac{n\alpha}{D}}\|x-y\|^{-1}t\xi\tau$. Using that $(1-s^2)(1-\xi^2)$ and $\xi$ are bounded by $1$, we obtain from the change of variables that
\begin{align*}
g_{Z_n}(t) &  \leqslant \frac{\lambda^{2n\beta}\omega_{D-2}^2}{D\P(x\in {\mathcal O}_{n,n})}\int_{0}^1\int_{\{\tau>0\}} \tau^D e^{-\lambda^{n\beta}\kappa_D\tau^D}\frac{(2\tau\xi)^D}{t}\frac{t\xi}{\lambda^{-\frac{n\alpha}{D}}\|x-y\|}\,\mathrm{d}\tau\,\mathrm{d}\xi \\
& \leqslant \frac{C}{\P(x\in {\mathcal O}_{n,n})}\frac{\lambda^{2n\beta+\frac{n\alpha}{D}}}{\|x-y\|}\int_{\{\tau>0\}}\tau^{2D}e^{-\lambda^{n\beta}\kappa_D\tau^D}\mathrm{d}\tau \\
& = \frac{C'}{\P(x\in {\mathcal O}_{n,n})}\frac{\lambda^{-\frac{n(\beta-\alpha)}{D}}}{\|x-y\|}
\end{align*}
where $C$ and $C'$ are two positive constants which only depend on $D$.
~\\~\\
\noindent {\it Subcase $D=2$}. We return to \eqref{eq:dernierelignemarchepourtoutD} and apply the same change of variables $s=s_{\tau}=\lambda^{\frac{n\alpha}{D}}\|x-y\|^{-1}t\xi\tau$. We now obtain that
\begin{equation}\label{eq:densiteD=2interm}
g_{Z_n}(t) \leqslant \frac{8\lambda^{2n\beta}\lambda^{\frac{n\alpha}{2}}}{\|x-y\|\P(x\in {\mathcal O}_{n,n})}\int_{\{\tau>0\}}\tau^4 e^{-\lambda^{n\beta}\pi\tau^2}\left(\cdots\right)
\mathrm{d}\tau
\end{equation}
where
$$\left(\cdots\right)
=\int_{\xi=0}^{1\wedge\frac{\lambda^{-\frac{n\alpha}{2}}\|x-y\|}{t\tau}}
(1-\xi)^{-\frac{1}{2}}\left(1-\frac{t\xi\tau}{\lambda^{-\frac{n\alpha}{2}}\|x-y\|}\right)^{-\frac{1}{2}}\mathrm{d}\xi.$$
We notice that there exists a positive constant $C>0$ such that for every $a>0$, we have
\begin{equation}\label{eq:lemmegrandiose}
\int_{\xi=0}^{1\wedge a}(1-\xi)^{-\frac{1}{2}}\bigg(1-\frac{\xi}{a}\bigg)^{-\frac{1}{2}}\mathrm{d}\xi\leqslant C|\log|a-1||.
\end{equation}
Indeed, due to the facts that the left-hand side of \eqref{eq:lemmegrandiose} is bounded for large $a$ and that the calculation is symmetric with respect to $1$,  it is enough to look for the behaviour of the Abelian-type integral when $a>1$ is close to $1$. A direct calculation shows then that
$$\int_0^1(1-\xi)^{-\frac{1}{2}}\left(1-\frac{\xi}{a}\right)^{-\frac{1}{2}}\mathrm{d}\xi=\sqrt{a}\;\mbox{argch}\left(\frac{a+1}{a-1}\right)\underset{a\to 1}{\sim}-\log(a-1),$$
which proves \eqref{eq:lemmegrandiose}. Consequently, we get from \eqref{eq:densiteD=2interm} and \eqref{eq:lemmegrandiose} that
\begin{align*}
g_{Z_n}(t)  & \leqslant \frac{C'\lambda^{2n\beta}\lambda^{\frac{n\alpha}{2}}}{\|x-y\|\P(x\in {\mathcal O}_{n,n})}\int_{\{\tau>0\}} \tau^4 e^{-\lambda^{n\beta}\pi\tau^2}\bigg|\log\bigg|1-\frac{\lambda^{-\frac{n\alpha}{2}}\|x-y\|}{t\tau}\bigg|\bigg|\;\mathrm{d}\tau,
\end{align*}
where $C'$ denotes again a positive constant.

\noindent We now fix $\varepsilon\in (0,1)$ and we split the integral:\\
- on the range of $\tau$ which satisfy  $|1-\frac{\lambda^{-\frac{n\alpha}{2}}\|x-y\|}{t\tau}|>\varepsilon$, the upper-bound is similar to the case $D\geqslant 3$;\\
- on the range of $\tau$ satisfying $|1-\frac{\lambda^{-\frac{n\alpha}{2}}\|x-y\|}{t\tau}|<\varepsilon$, the integral is bounded by
$$\frac{1}{(1-\varepsilon)^6}e^{-\lambda^{n\beta}\pi\frac{\lambda^{-n\alpha}\|x-y\|^2}{(1+\varepsilon)^2t^2}}
\left(\frac{\lambda^{-\frac{n\alpha}{2}}\|x-y\|}{t}\right)^5 \int_{u=1-\varepsilon}^{1+\varepsilon}|\log|1-u^{-1}|\mathrm{d}u\leqslant C' \varphi\left(\frac{\lambda^{-\frac{n\alpha}{2}}\|x-y\|}{t}\right)$$
where $\varphi(u)=e^{-\lambda^{n\beta}\pi\frac{u^2}{(1+\varepsilon)^2}}u^5$, $u>0$.

\noindent It remains to notice that the maximum of the function $\varphi$ is of order $\mathrm{O}(\lambda^{-\frac{5}{2}n\beta})$ to deduce the required result \eqref{eq:norminfdenshn}.
\end{proof}

\noindent We conclude this subsection with the integrability of the Lipschitz constant $L(x)$ of the affine part of $\Delta_0$ above $x$, i.e.
\begin{equation}\label{eq:lipvoisins}
L(x)=\frac{2}{\|c_0(x)-c_0'(x)\|}.
\end{equation}
We define the new set $\widetilde{\mathcal{O}}_{n,N}$ as
$$\widetilde{\mathcal{O}}_{n,N}=\bigg\{x\in [0,1]^D : \mbox{all points of $B_{\lambda^{\frac{n\beta}{D}}\tau_N}(x)$ are in the same simplex of ${\mathcal S}_0$ as $x$} \bigg\}.$$

\begin{prop}\label{prop:lipconst}
For every $x\in [0,1]^D$, $\E(L(x))<\infty$ and $\displaystyle\sup_{0\leqslant n\leqslant N}\E(L(x)|\,x\in \widetilde{\mathcal{O}}_{n,N})<\infty$.
\end{prop}

\begin{proof}
We could deal with the conditional distribution of $L(x)$ in the same spirit as in the proof of Proposition \ref{prop:denshn}. The conditional density of $\|c_0(x)-c_0'(x)\|$ would be in particular very close to \eqref{densityLn}. For sake of simplicity, we choose to use a direct argument for removing the conditioning. Indeed, we notice the following fact: on the event $\{x\in \widetilde{\mathcal{O}}_{n,N}\}$, a vicinity of $x$ is in the same simplex of ${\mathcal S}_0$ which means that the conditioning favors flatter pyramid faces above $x$ and smaller Lipschitz constants $L(x)$. Consequently, we have $\E(L(x)\,|\,x\in \widetilde{\mathcal O}_{n,N})\leqslant\E(L(x))$ for every $n,N\geqslant 0$.

\noindent For $D=1$, the integrability of the variable $L(x)$ given by \eqref{eq:lipvoisins} comes from the fact that the distance from the two neighbors of $x$ (the nearest and second nearest) is Gamma-distributed. When $D\geqslant 2$, we use a reasoning similar to the proof of Proposition \ref{prop:denshn} to obtain that
\begin{align}
\E(L(x)) & =\E\biggl(\,\sum_{z_1\ne z_2} \frac{2}{\|z_1-z_2\|}\indic_{(c_n(x),c_n'(x))}(z_1,z_2)\biggl)\nonumber\\
 & =\iint \frac{2}{\|z_1-z_2\|}\P((x+\R_+(x-z_1))\cap H_{z_1,z_2}\in {\mathcal C}_{z_1}\cap {\mathcal C}_{z_2}) \mathrm{d}z_1\mathrm{d}z_2.\label{eq:calculL(x)}
\end{align}
We write $z_2=z_1+\rho u$ where $u\in{\mathbb S}^{D-1}$, $\rho>0$ and $u=s\frac{x-z_1}{\|x-z_1\|}+\sqrt{1-s^2}v$
where $s\in (0,1)$ and $v\in {\mathbb S}^{D-1}\cap \{x-z_1\}^{\perp}$. In particular, the distance from $z_1$ to the point $(x+\R_+(x-z_1))\cap H_{z_1,z_2}$ is $\frac{\rho}{2s}$. Consequently, we deduce from a change of variables applied to the integral in \eqref{eq:calculL(x)} that
$$
\E(L(x))
=\iint\int_{s=0}^1\frac{2\omega_{D-2}}{\rho}
{\indic}_{\R_+}(\rho-2\|x-z_1\|s)
e^{-\kappa_D\left(\frac{\rho}{2s}\right)^D}\rho^{D-1}(1-s^2)^{\frac{D-3}{2}}\mathrm{d}s\,\mathrm{d}\rho\,\mathrm{d}z_1.
$$
When $D\geqslant 3$, we proceed with the change of variables $\tau=\tau_s=\frac{\rho}{2s}$. There is a constant $C>0$ such that
\begin{align*}
\E(L(x)) &\leqslant \omega_{D-2}\int_{\rho=0}^{\infty}\rho^{D-1}\int_{\tau=\frac{\rho}{2}}^{\infty}\bigg(\int_{B(x,\tau)}\mathrm{d}z_1\bigg)e^{-\kappa_D \tau^D}\tau^{-2}\mathrm{d}\tau\mathrm{d}\rho\\
& \leqslant C\int_{\rho=0}^{\infty}\rho^{D-1}\int_{\tau=\frac{\rho}{2}}^{\infty}\tau^{D-2}e^{-\kappa_D\tau^D}\mathrm{d}\tau\mathrm{d}\rho\\
& \leqslant C\int_{\rho=0}^{\infty}\rho^{D-1}e^{-\frac{\kappa_D}{2}\left(\frac{\rho}{2}\right)^D}\mathrm{d}\rho
\;\int_{\tau=0}^{\infty}\tau^{D-2}e^{-\frac{\kappa_D}{2}\,\tau^D}\mathrm{d}\tau<\infty.
\end{align*}
Finally, when $D=2$, with the same change of variables, we get
\begin{equation}\label{eq:constlipcasdim2}
\E(L(x))\leqslant C\int_{\rho=0}^{\infty}\rho\int_{\tau=\frac{\rho}{2}}^{\infty}e^{-\pi\tau^2}\frac{\mathrm{d}\tau\mathrm{d}\rho}{\sqrt{1-\left(\frac{\rho}{2\tau}\right)^2}}.
\end{equation}
We treat separately the integral in $\tau$ for fixed $\rho>0$ :
\begin{align}
\int_{\tau=\frac{\rho}{2}}^{\infty}e^{-\pi\tau^2}\frac{\mathrm{d}\tau}{\sqrt{1-\left(\frac{\rho}{2\tau}\right)^2}}
& \leqslant 2\rho e^{-\pi\left(\frac{\rho}{2}\right)^2}\int_{\tau=\frac{\rho}{2}}^{\rho}\frac{\rho}{2\tau^2}\frac{\mathrm{d}\tau}{\sqrt{1-\left(\frac{\rho}{2\tau}\right)^2}}
+\sqrt{\frac{4}{3}}\int_{\tau=\rho}^{\infty}e^{-\pi\tau^2}\mathrm{d}\tau\nonumber\\
& \leqslant 2\rho e^{-\pi\left(\frac{\rho}{2}\right)^2}\left[\arccos\left(\frac{\rho}{2u}\right)\right]_{\frac{\rho}{2}}^{\rho}
+\sqrt{\frac{4}{3}}e^{-\frac{\pi}{2}\rho^2}\int_{\tau=0}^{\infty}e^{-\frac{\pi}{2}\tau^2}\mathrm{d}\tau\nonumber\\
& \leqslant \frac{2\pi}{3}\rho e^{-\pi\left(\frac{\rho}{2}\right)^2} + C'e^{-\frac{\pi}{2}\rho^2},\label{eq:intapart}
\end{align}
where $C'$ is a positive constant. Inserting \eqref{eq:intapart} in \eqref{eq:constlipcasdim2}, we get the required result.
\end{proof}

\noindent

\subsection{Size of the increments of $F$}\label{subsec:inccond}

\pass

\noindent The box-dimension of $\Gamma$, as well as its Hausdorff dimension, is closely related to the oscillations of $F$ (see \cite{dubtri,fal03}). Let us recall that, for every $A\subset [0,1]^D$, the oscillation of $F$ over $A$ is defined by
\begin{equation}\label{eq:defoscA}
\osc(F,A)=\sup_{y,y'\in A} |F(y')-F(y)|.
\end{equation}
In particular, we will consider, for all $x\in[0,1]^D$ and $\tau>0$, the oscillation of $F$ over the cube $x+[0,\tau]^D$ given by \begin{equation}\label{eq:defosc}
\osc_{\tau}(x)=\osc(F,x+[0,\tau]^D)=\sup_{y,y'\in \, x+[0,\tau]^D} |F(y')-F(y)|.
\end{equation}

\begin{prop}\label{prop:estimoscill}
Let $0<p<\frac{\alpha}{H}$. Then, for every $x\in[0,1]^D$, we have, when $N\to \infty$,
\begin{equation*}\label{eq:estimoscill}
\E(\osc_{\tau_N}(x))=\E(\osc_{\tau_N}(0))= \mathrm{O}(\tau_N^p).
\end{equation*}
\end{prop}

\begin{proof}
Let us write
\begin{equation*}
\delta_N=\sup_{x\in [0,\tau_N]^D} |F(x)-F(0)|. 
\end{equation*}

\noindent We claim that
\begin{equation}\label{eq:resultBC}
\P\left(\liminf \left\{\delta_N\leqslant \tau_N^p\right\}\right)=1.
\end{equation}

\noindent By Markov's inequality,
\begin{equation}\label{eq:majmarkov}
\P\left(\delta_N\geqslant \tau_N^p\right)\leqslant \tau_N^{-p}\,\E\left(\delta_N\right).
\end{equation}

\noindent We can write
\begin{equation*}
\E\left(\delta_N\right)\leqslant \E\biggl(\,\sum_{n=0}^{\infty}\,\sup_{x\in [0,\tau_N]^D} |Z_n(x,0)|\biggl) = S_1(N)+S_2(N)+S_3(N)
\end{equation*}
with
\begin{align*}
& S_1(N) = \sum_{n=0}^N\E\biggl(\,\sup_{x\in [0,\tau_N]^D}|Z_n(x,0)|\,\biggl|\,0\in {\mathcal O}_{n,N}\biggl)\P(0\in {\mathcal O}_{n,N}), \\
& S_2(N) = \sum_{n=0}^{N}\E\biggl(\,\sup_{x\in [0,\tau_N]^D}|Z_n(x,0)|\,\biggl|\,0\not\in {\mathcal O}_{n,N}\biggl)\P(0\not\in {\mathcal O}_{n,N}), \\
& S_3(N) = \sum_{n=N+1}^{\infty}\E\biggl(\,\sup_{x\in [0,\tau_N]^D}|Z_n(x,0)|\biggl).
\end{align*}
Since $[0,\tau_n]^D\subset B_{\sqrt{D}\tau_N}(0)$, we mention that for the purpose of this proof, the definition \eqref{eq:randomset} of the set ${\mathcal O}_{n,N}$ should be slightly adapted by substituting $\sqrt{D}\tau_N$ for $\tau_N$. For sake of simplicity, we omit that technical detail.

\noindent For $S_1(N)$, we notice that
\begin{equation*}
\E\biggl(\,\sup_{x\in [0,\tau_N]^D}|Z_n(x,0)|\,\biggl|\,0\in {\mathcal O}_{n,N}\biggl)\leqslant (\lambda^{-\frac{n\alpha}{D}}\sqrt{D}\,\tau_N)(\lambda^{\frac{n\beta}{D}}\,\E(L(0)\,|\,0\in \widetilde{{\mathcal O}}_{n,N}))
\end{equation*}
where $L(0)$ is the Lipschitz constant of $\Delta_0$ at $0$ when the underlying Poisson point process is homogeneous of intensity $1$. Indeed, the function $\Delta_n$ is Lipschitz above $0$  with Lipschitz constant equal in distribution to $\lambda^{\frac{n\beta}{D}}L(0)$ by scaling invariance. Moreover, the distribution of that Lipschitz constant conditional on $\{0\in {\mathcal O}_{n,N}\}$ is the same as the distribution of $\lambda^{\frac{n\beta}{D}}L(0)$ conditional on $\{0 \in \widetilde{{\mathcal O}}_{n,N}\}$. Using \eqref{eq:defZn}, we get the inequality above. Thanks to Proposition \ref{prop:lipconst},
we have $\sup_{n,N\in\N}\E(L(0)\,|\,0\in \widetilde{{\mathcal O}}_{n,N})<\infty$.

\noindent Thus, using $0<\alpha\leqslant\beta\leqslant 1$, we obtain
\begin{equation*}
S_1(N) \leqslant \sum_{n=0}^N \lambda^{-\frac{n\alpha}{D}}\lambda^{-\frac{NH}{D}}\lambda^{\frac{n\beta}{D}}\,\sqrt D\,\sup_{n,N\in\N}\E(L(0)\,|\,0\in \widetilde{{\mathcal O}}_{n,N}) \leqslant \left\{
                      \begin{array}{ll}
                       C_1\,\lambda^{\frac{(\beta-\alpha-H)N}{D}} & \hbox{if $\alpha<\beta$} \\
                       C'_1\,N\lambda^{\frac{-NH}{D}} & \hbox{if $\alpha=\beta$}
                      \end{array}
                    \right.
\end{equation*}
where $C_1,C'_1$ are two positive constants which do not depend on $N$.

\noindent For $S_2(N)$ we use the upper estimate (see \eqref{eq:defZn})
\begin{equation*}
\E\biggl(\,\sup_{x\in [0,\tau_N]^D}|Z_n(x,0)|\,\biggl|\,0\not\in {\mathcal O}_{n,N}\biggl)\leqslant \lambda^{-\frac{n\alpha}{D}}
\end{equation*}
and Proposition \ref{prop:size}(i) to get
\begin{equation*}
S_2(N) \leqslant \sum_{n=0}^N \lambda^{-\frac{n\alpha}{D}}\lambda^{\frac{n\beta-NH}{D}}\leqslant \left\{
                      \begin{array}{ll}
                       C_2\,\lambda^{\frac{(\beta-\alpha-H)N}{D}} & \hbox{if $\alpha<\beta$} \\
                       C'_2\,N\lambda^{\frac{-NH}{D}} & \hbox{if $\alpha=\beta$}
                      \end{array}
                    \right.
\end{equation*}
where $C_2,C'_2$ are two positive constants which do not depend on $N$.

\noindent Finally, for $S_3(N)$ we only use the upper estimate
\begin{equation*}
\E\biggl(\,\sup_{x\in [0,\tau_N]^D}|Z_n(x,0)|\biggl)\leqslant \lambda^{-\frac{n\alpha}{D}}
\end{equation*}
to get
\begin{align*}
S_3(N) \leqslant \sum_{n=N+1}^{\infty} \lambda^{-\frac{n\alpha}{D}} \leqslant C_3\, \lambda^{-\frac{N\alpha}{D}}
\end{align*}
where $C_3>0$ is a positive constant which does not depend on $N$.

\noindent Therefore,
\begin{equation*}
\E\biggl(\,\sum_{n=0}^{\infty}\,\sup_{x\in [0,\tau_N]^D} |Z_n(x,0)|\biggl) \leqslant \left\{
                      \begin{array}{ll}
                       C\bigl(\tau_N^{1-\frac{\beta-\alpha}{H}} + \tau_N^{\frac{\alpha}{H}}\bigl) & \hbox{if $\alpha<\beta$} \\
                       C'\bigl(|\log(\tau_N)|\tau_N +\tau_N^{\frac{\alpha}{H}}\bigl) & \hbox{if $\alpha=\beta$}
                      \end{array}
                    \right.
\end{equation*}
where $C,C'$ are two positive constants which do not depend on $N$. The rhs of \eqref{eq:majmarkov} is in particular summable in $N$ as soon as $p<\frac{\alpha}{H}$ (which guarantees that $1-\frac{\beta-\alpha}{H}>p$ since $1-\frac{\beta}{H}>0$). Consequently, by Borel-Cantelli's lemma, \eqref{eq:resultBC} holds. Then, we obtain
\begin{equation*}
\osc_{\tau_N}(0) \leqslant 2\delta_N = \mathrm{O}(\tau_N^p)
\end{equation*}
almost surely for all $N$ large enough. To conclude, let us notice that, by stationarity, it follows that
\begin{equation*}
\E(\osc_{\tau_N}(x))=\E(\osc_{\tau_N}(0))= \mathrm{O}(\tau_N^p)
\end{equation*}
for all $x\in[0,1]^D$.
\end{proof}

\noindent We conclude this subsection with an estimate of the expectation of a particular functional of the increment $F(x)-F(y)$ that will appear in the application of a Frostman-type lemma in the next section.

\begin{prop}\label{prop:majespercond}
Let $s>1$ and $n\geqslant 0$. If $x,y\in[0,1]^D$ satisfy $\tau_{n+1}<\|x-y\|\leqslant\tau_n$ then
\begin{equation}\label{eq:esperincZ}
\E\left((|F(x)-F(y)|^2+\|x-y\|^2)^{-\frac{s}{2}}\indicat_{{\mathcal O}_{n,n}}(x)\right) \leqslant C \|x-y\|^{-s+\frac{\beta-\alpha}{H}}
\end{equation}
where $C>0$ is a constant which does not depend on $x$, $y$, $n$.
\end{prop}

\begin{proof}
Remember that $F(x)-F(y)=Z_n(x,y)+S_n(x,y)$ where $Z_n(x,y)$ and $S_n(x,y)$ are independent. Let $\P_n$ be the probability associated with ${\mathcal X}_n$ and $\mu_{S_n}$ be the probability distribution of the random variable $S_n(x,y)$. From Proposition \ref{prop:denshn} one obtains
\begin{align*}
\E &\left((|F(x)-F(y)|^2 + \|x-y\|^2)^{-\frac{s}{2}}\indicat_{{\mathcal O}_{n,n}}(x)\right) \\
& = \iint ((Z_n(x,y)+v)^2+\|x-y\|^2)^{-\frac{s}{2}} \indicat_{{\mathcal O}_{n,n}}(x) \dd \P_n \dd \mu_{S_n}(v)\\
& = \iint \P(x\in {\mathcal O}_{n,n})((u+v)^2+\|x-y\|^2)^{-\frac{s}{2}} g_{Z_n}(u)\dd u \dd \mu_{S_n}(v) \\
& \leqslant \int\P(x\in {\mathcal O}_{n,n})\biggl(\int_{\{|u+v|<\|x-y\|\}}\hspace{-1cm}\|x-y\|^{-s}g_{Z_n}(u)\dd u \dd \mu_{S_n}(v)
+ \int_{\{|u+v|>\|x-y\|\}}\hspace{-1cm}|u+v|^{-s}g_{Z_n}(u)\dd u \dd \mu_{S_n}(v)\biggl) \\
& \leqslant 2 \|x-y\|\,\sup_{t\in\R}(g_{Z_n}(t)) \|x-y\|^{-s}+ \sup_{t\in\R}(g_{Z_n}(t)) \int_{\{|u+v|>\|x-y\|\}}\hspace{-1cm}|u+v|^{-s}\dd u \dd \mu_{S_n}(v) \\
& \leqslant 2 C\|x-y\|\,\|x-y\|^{-1}\lambda^{-\frac{n(\beta-\alpha)}{D}}\,\|x-y\|^{-s}+ C\|x-y\|^{-1}\lambda^{-\frac{n(\beta-\alpha)}{D}}\,\|x-y\|^{-s+1} \\
& =  C \|x-y\|^{-s}\lambda^{-\frac{n(\beta-\alpha)}{D}}.
\end{align*}

\noindent Finally, the assumption on $\|x-y\|$ implies that $\lambda^{-n}\leqslant (\|x-y\|\lambda^{\frac{H}{D}})^{\frac{D}{H}}$, which provides the desired bound.
\end{proof}

\section{Proof of the main theorem}\label{sec:proofmain}

\noindent Let us recall that for any non-empty compact set $K\subset\R^{D+1}$ one has (see \cite{fal03})
\begin{equation}\label{eq:ineqdim}
0\leqslant \dimhaus(K)\leqslant\dimbox(K)\leqslant D+1.
\end{equation}
Thus the proof of Theorem \ref{theo:maintheo} will consist in proving that $D+1-\frac{\alpha}{\beta}$ is an upper bound for $\dimbox(\Gamma)$ and a lower bound for $\dimhaus(\Gamma)$.

\subsection{An upper bound for the box-dimension of $\Gamma$}\label{sec:boxdim}

\pass

\noindent First we investigate the box-dimension of $\Gamma$. For every $\tau>0$ we cover $[0,1]^D\times \R$ with $\tau$-mesh cubes and denote by ${\mathcal N}(\tau)$ the (finite) number of cubes from this partition which intersect $\Gamma$. Then, we can express the box-dimension of $\Gamma$ in terms of ${\mathcal N}(\tau_N)$ (see Section 3.1 in \cite{fal03} and Section 2.2 in \cite{tri95}):
\begin{equation}\label{eq:boxdimbox}
\dimbox(\Gamma)=\underset{N\to \infty}{\limsup}\,\frac{\log {\mathcal N}(\tau_N)}{|\log(\tau_N)|}.
\end{equation}

\pass
\begin{lem}\label{lem:dimconst}
There exists a constant $d>0$ such that $\,\,\P(\dimbox(\Gamma)= d)=1$.
\end{lem}

\begin{proof}
For every $m\in\N$, let us denote by $\cal{A}_m$ the $\sigma$-algebra generated by the point process $\cal{X}_m$ and $F_m(x)=\sum_{n=0}^m \lambda^{-\frac{n\alpha}{D}}\Delta_n(x)$, $x\in [0,1]^D$. Since $F_m$ is a Lipschitz function, the graph of the function $F-F_m$ has the same box-dimension as the graph of $F$ (see Section 12.4 in \cite{tri95} or Chapter 11 in \cite{fal03}). Consequently, we can use \eqref{eq:boxdimbox} to show that $\dimbox(\Gamma)$ is a random variable which is measurable with respect to $\sigma(\cal{A}_m:m\geqslant n)$ for every $n\in\N$. We then use the $0$-$1$ law to deduce that it is almost surely constant.
\end{proof}

\noindent\textbf{Proof of the upper bound of \eqref{eq:maintheo}.}

\noindent Let $(\tau_{N_k})_k$ be a subsequence of $(\tau_N)_N$ such that
\begin{equation*}
\dimbox(\Gamma)=\lim_{k\to\infty}\frac{\log\cal{N}(\tau_{N_k})}{|\log(\tau_{N_k})|}\leqslant D+1.
\end{equation*}
Applying Lebesgue's convergence theorem we obtain
\begin{equation*}
\E\bigg(\lim_{k\to\infty}\frac{\log \cal{N}(\tau_{N_k})}{|\log(\tau_{N_k})|}\bigg)
=\lim_{k\to\infty}\E\bigg(\frac{\log \cal{N}(\tau_{N_k})}{|\log(\tau_{N_k})|}\bigg)
\leqslant\limsup_{N\to\infty}\E\bigg(\frac{\log \cal{N}(\tau_{N})}{|\log(\tau_{N})|}\bigg).
\end{equation*}
Then, Jensen's inequality with the concave function $\log$ yields
\begin{align*}
\E(\dimbox(\Gamma))
\leqslant \limsup_{N\to\infty}\E\bigg(\frac{\log \cal{N}(\tau_{N})}{|\log(\tau_{N})|}\bigg)
\leqslant \limsup_{N\to\infty}\log\E\biggl(\frac{{\mathcal N}(\tau_N)}{|\log(\tau_N)|}\biggl).
\end{align*}
We now use Proposition $11.1$ from \cite{fal03} and Proposition \ref{prop:estimoscill} to get
\begin{equation*}
\E({\mathcal N}(\tau_N))\leqslant 2\lceil \tau_N^{-1}\rceil^D + \tau_N^{-1} \sum_{\substack{\mathrm{k}=(k_1,\ldots,k_D)\\ 0\leqslant k_i \leqslant\lceil \tau_N^{-1}\rceil}}\E(\osc_{\tau_N}(\mathrm{k}\tau_N)) = \mathrm{O}\bigl(\tau_N^{p-D-1}\bigl).
\end{equation*}
In conclusion, for every $H>\beta$ and $p<\frac{\alpha}{H}$, we obtain $\E(\dimbox(\Gamma))\leqslant D+1-p$. We finally get
\begin{equation*}
\dimbox(\Gamma)=\E(\dimbox(\Gamma))\leqslant D+1-\dfrac{\alpha}{\beta}
\end{equation*}
by letting $H\to\beta$, $p\to \frac{\alpha}{\beta}$, and by using Lemma \ref{lem:dimconst}. \hfill $\square$

\subsection{A lower bound for the Hausdorff dimension of $\Gamma$}\label{sec:hausdim}

\pass

\noindent To find a lower bound for the Hausdorff dimension of a compact set is generally a difficult problem. An important step was made in \cite{hunt98} when Hunt proposed a way to find a lower bound for the Hausdorff dimension of the graph of Weierstrass functions using the `finite energy criterion'. The arguments of \cite{hunt98} can be applied for a large class of Weierstrass-type functions, but not for the Takagi-Knopp series defined by \eqref{def:taka1D} because the sawtooth function $\Delta$ is not regular enough. This well-known criterion is used for calculating the Hausdorff dimension of more general fractal sets (see e.g. Chapter 4 in \cite{mope}).

\noindent Let us recall that the Hausdorff dimension of a non-empty compact set $K\subset\R^{D+1}$ may be expressed in terms of finite energy of some measures thanks to a lemma due to Frostman (see e.g. \cite{kah85,mat95}):
\begin{equation*}\label{eq:hausenergy}
\dimhaus(K)= \sup_{\mu}\,\bigl(\sup\{s\geqslant 0 : I_s(\mu)<\infty\}\bigl)
\end{equation*}
where the supremum on $\mu$ is taken over all the finite and non-null Borel measures such that $\mu(K)>0$, $I_s(\mu)$ being the $s$-energy of $\mu$ defined by
\begin{equation}\label{eq:muenergy}
I_s(\mu)= \iint \|x-y\|^{-s} \dd\mu(x)\dd\mu(y).
\end{equation}
Therefore, if such a measure $\mu$ satisfies $I_s(\mu)<\infty$ then $\dimhaus(K)\geqslant s$.

\pass Let us recall now a classical way to construct such a measure on $\Gamma$. For each $N\geqslant 0$ the set $W_N$ (see \eqref{eq:WN}) is a Borel subset of $[0,1]^D$. Since $F$ is a continuous function then $\Gamma$ is a Borel set too so that we can consider the measure $\mu_{W_N}$ obtained by lifting onto $\Gamma$ the $D$-dimensional Lebesgue measure restricted to $W_N$. Precisely, for all Borel set $E\subset\R^{D+1}=\R^D\times\R$,
\begin{equation*}\label{eq:defmu}
\mu_{W_N}(E)= \vol\bigl\{x\in[0,1]^D\cap W_N \text{ such that }(x,F(x))\in E \bigl\}
\end{equation*}
and $\mu_{W_N}$ is a positive measure as soon as $\vol(W_N)>0$. The $s$-energy of $\mu_{W_n}$ is then
\begin{equation}\label{eq:muenergybis}
I_s(\mu_{W_N})= \iint_{W_N\times W_N} \bigl(\|x-y\|^2+|F(x)-F(y)|\big)^2)^{-\frac{s}{2}}\dd x\dd y.
\end{equation}
Notice that the finiteness of $I_s(\mu_{W_N})$ depends only on the size of the increments $F(x)-F(y)$ when $\|x-y\|$ is small.

\pass\textbf{Proof of the lower bound of \eqref{eq:maintheo}.}

\noindent Let us consider, for all $n\geqslant0$, the set
$$ T_n = \left\{(x,y)\in [0,1]^D \times [0,1]^D : \tau_{n+1} <\|x-y\|\leqslant \tau_n \right\}.$$

\noindent Let $N\geqslant 1$ and $s>1$. Since $W_N\subset {\mathcal O}_{n,n}$ for all $n\geqslant N$, we have, with \eqref{eq:muenergybis},
\begin{align*}
I_s(\mu_{W_N}) & \leqslant \iint_{(x,y)\in W_N\times [0,1]^D} \bigl(\|x-y\|^2+|F(x)-F(y)|\big)^2)^{-\frac{s}{2}}\,\dd x\dd y \\
& \leqslant C'+C_N +  \sum_{n=N}^{\infty} \iint_{(x,y)\in T_n \cap (W_N\times [0,1]^D)} \bigl(\|x-y\|^2+|F(x)-F(y)|\big)^2)^{-\frac{s}{2}}\,\dd x\dd y \\
& \leqslant C'+C_N +  \sum_{n=N}^{\infty} \iint_{(x,y)\in T_n} \bigl(\|x-y\|^2+|F(x)-F(y)|\big)^2)^{-\frac{s}{2}}\indic_{{\mathcal O}_{n,n}}(x)\,\dd x\dd y
\end{align*}
where
\begin{align*}
C' & = \iint_{\{(x,y)\in W_N\times [0,1]^D : \|x-y\|\geqslant 1\}} \bigl(\|x-y\|^2+|F(x)-F(y)|\big)^2)^{-\frac{s}{2}}\,\dd x\dd y \\
& \leqslant \iint_{\{(x,y)\in [0,1]^D\times [0,1]^D : \|x-y\|\geqslant 1\}}\|x-y\|^{-s} \dd x\dd y \leqslant 1,
\end{align*}
and
\begin{align*}
C_N & = \iint_{(x,y)\in (T_0\cup \cdots \cup T_{N-1}) \cap (W_N\times [0,1]^D)} \bigl(\|x-y\|^2+|F(x)-F(y)|\big)^2)^{-\frac{s}{2}}\,\dd x\dd y \\
& \leqslant \iint_{(x,y)\in (T_0\cup \cdots \cup T_{N-1})} \|x-y\|^{-s} \dd x\dd y \leqslant \lambda^{NHs}.
\end{align*}

\noindent To show that the integral \eqref{eq:muenergybis} is finite almost surely it is enough to show that its expectation is finite. By Fubini's theorem,
\begin{equation*}\label{eq:majintenergy}
\E(I_s(\mu_{W_N})) \leqslant 1+\lambda^{NHs}
+ \sum_{n=N}^{\infty}\iint_{(x,y)\in T_n} \E\bigl(\bigl(\|x-y\|^2+|F(x)-F(y)|\big)^2)^{-\frac{s}{2}}\indic_{{\mathcal O}_{n,n}}(x)\bigl)\dd x\dd y.
\end{equation*}

\noindent Using the estimate \eqref{eq:esperincZ} we obtain
\begin{align*}\label{eq:majintenergybis}
\E(I_s(\mu_{W_N})) \leqslant 1+\lambda^{NHs} + C\iint_{(x,y)\in \cup_{n=N}^{\infty} T_n} \|x-y\|^{-s+\frac{\beta-\alpha}{H}}\dd x\dd y.
\end{align*}

\noindent This latter integral converges as soon as $-s+\frac{\beta-\alpha}{H}>-D$. Therefore the random measure $\mu_{W_N}$ has a finite $s$-energy
for all $1<s<D+\frac{\beta-\alpha}{H}$. Since $N$ may be chosen such that the probability $\P(\vol(W_N)>0)$ will be arbitrarily close to $1$ (see Proposition \ref{prop:size}) we deduce that
\begin{equation*}
\dimhaus(\Gamma)\geqslant D+\frac{\beta-\alpha}{H}
\end{equation*}
for all $H>\beta$ and almost surely. We obtain the desired lower bound by letting $H\in\Q_+$ go to $\beta$. \hfill $\square$

\section{Related models}

\noindent In this section, we study the fractal properties of two different models which are related to our Poisson-Voronoi construction: a deterministic series of pyramidal functions with hexagonal bases on the one hand, a random perturbation of the classical Takagi-Knopp series on a dyadic mesh on the other hand.

\subsection{Takagi-like series directed by hexagonal Voronoi tessellations}\label{sec:takahexa}

\pass

\noindent The series that we study here is only defined in $\R^2$. It is very close to the original function $F_{\lambda,\alpha,\beta}$. The novelty lies in the $\Delta_n$ functions: we consider now pyramids with a regular hexagonal basis. This model is naturally related to the previous one for two reasons. First, an hexagonal mesh is known to be the Voronoi tessellation generated by a regular triangular mesh. Secondly, the mean of the number of vertices of a typical cell from a Poisson-Voronoi tessellation is known to be $6$ (see e.g. Prop. 3.3.1. in \cite{mol94}) so that an hexagonal mesh may be seen as an {\it idealized} realization of a Poisson-Voronoi tessellation.

\pass We start with the deterministic Voronoi tessellation whose cells are identical regular hexagons such that one is centered at the origin $(0,0)$ and has a vertex at $(1,0)$. Then, considering all the centers of these hexagons as a set of points $\cal{X}_0$, we set $\cal{X}_n = 2^{-n}\cal{X}_0$ and construct the hexagonal Voronoi tessellation associated with. Here again $\Delta_n:\R^2 \longrightarrow [0,1]$ is the piecewise linear pyramidal function satisfying $\Delta_n=0$ on $\bigcup_{c\in\cal{X}_n}\partial \cal{C}_c$ and $\Delta_n=1$ on $\cal{X}_n$.

\noindent Let us notice that, for all $x\in\R^2$ and all $n\geqslant 0$, we have $\Delta_n(x)=\Delta(2^n x)$. We fix $\alpha\in(0,1]$ and define a function
\begin{equation}\label{def:takahexa}
f_{\alpha}(x) = \sum_{n=0}^{\infty} 2^{-n\alpha} \Delta_n(x) = \sum_{n=0}^{\infty} 2^{-n\alpha} \Delta(2^n x)\,\,,\,x\in\R^2.
\end{equation}
The main theorem of this section is the analogue of Theorem \ref{theo:maintheo}. Actually we state a more precise result than Proposition \ref{prop:estimoscill} for the oscillations of $f_{\alpha}$ but we cannot determine the Hausdorff dimension of $\Gamma_{\alpha}$.

\begin{theo}\label{theo:maintheotakahexa}
Let $0<\alpha\leqslant 1$. Then $f_{\alpha}$ is a continuous function such that
\begin{equation}\label{eqtheo:takaosc}
\exists\,C,C'>0,\quad\forall\,\tau\in(0,1),\quad\forall\,x\in[0,1]^2,\quad C'\tau^{\alpha}\leqslant \osc_{\tau}(x) \leqslant C\tau^{\alpha}.
\end{equation}
Moreover, its graph
\begin{equation*}
\Gamma_{\alpha}=\left\{(x,f_{\alpha}(x)):x\in[0,1]^2\right\}\subset \R^2\times \R
\end{equation*}
is a fractal set satisfying
\begin{equation}\label{eq:maintheotakahexa}
\dimbox(\Gamma_{\alpha})=3-\alpha.
\end{equation}
\end{theo}

\begin{proof} In the sequel we drop again the index $\alpha$ so that $f=f_{\alpha}$ and $\Gamma=\Gamma_{\alpha}$. We also keep the notation $Z_n(x,y) = 2^{-n\alpha}(\Delta_n(x)-\Delta_n(y)) = 2^{-n\alpha}(\Delta(2^n x)-\Delta(2^n y))$ for all $x,y\in\R^2$.

\vspace{0.15cm}
\noindent (i) Let us state the upper estimates first. We fix $x,y\in[0,1]^2$ such that $\|x-y\|\in(0,1)$ and consider $N\geqslant 1$ such that $2^{-N}<\|x-y\|\leqslant 2^{-(N-1)}$. Using the fact that $\Delta$ is Lipschitz, with Lipschitz constant $1$, and bounded by $1$, we have $|Z_n(x,y)|\leqslant 2^{-n\alpha}\min(2^n\|x-y\|,1)$. Therefore
\begin{align*}
|f(x)- f(y)| \leqslant \sum_{n=0}^{\infty} |Z_n(x,y)| & \leqslant \sum_{n=0}^{N-1} (2^{n(1-\alpha)}\|x-y\|) + \sum_{n=N}^{\infty} 2^{-n\alpha} \\
& \leqslant \|x-y\|\frac{2^{N(1-\alpha)}}{2^{1-\alpha}-1} + \frac{2^{-N\alpha}}{1-2^{-\alpha}} \\
& \leqslant \|x-y\|\frac{(2\|x-y\|^{-1})^{(1-\alpha)}}{2^{1-\alpha}-1} + \frac{\|x-y\|^{\alpha}}{1-2^{-\alpha}} \\
& \leqslant C\|x-y\|^{\alpha}
\end{align*}
where $C$ is a positive constant which only depends on $\alpha$. Then, $|f(x)-f(y)| \leqslant C \tau^{\alpha}$ for all $x,y\in\R^2$ such that $0<\|x-y\|<\sqrt2\tau<1$. This gives the upper bound in \eqref{eqtheo:takaosc}.

\pass (ii) Now we state the lower estimates. We begin with finding a lower bound for the oscillation over an hexagonal cell. The key-point for estimating this oscillation is to calculate two particular increments for two pairs of well-chosen points belonging to the cell, namely the center and vertices. To keep in mind the number of the generation, we denote by $\cal{C}_c^n$, $c\in\cal{X}_n$ (resp. ${\mathcal T}_n$, ${\mathcal S}_n$, $\sk_n$) the cells of the Voronoi tessellation of generation $n$ (resp. the set $\{\cal{C}_c^n:c\in\cal{X}_n\}$, the associated simplex tessellation and the skeleton of the simplex tessellation). Moreover, the six vertices of a cell $\cal{C}_c^n$ are denoted by $c_i$ with $i\in\{1,\ldots,6\}$.

\noindent Then, let $N\geqslant 1$, $c\in\cal{X}_N$, $c_1,\ldots,c_6$ the vertices of $\cal{C}_c^N$ and $i\in\{1,\ldots,6\}$. We can write
\begin{equation*}
f(c)-f(c_i) = \sum_{n=0}^{N-1} Z_n(c,c_i) + \sum_{n=N}^{\infty} Z_n(c,c_i).
\end{equation*}

\noindent As soon as a point is the center (resp. a vertex) of a cell $\cal{C}_c^N$ then it is the center (resp. a vertex) of all the cells of higher generations. Hence $\Delta(2^n c)=1$ and $\Delta(2^n c_i)=0$ for all $n\geqslant N$. Therefore
\begin{equation*}
\sum_{n=N}^{\infty} Z_n(c,c_i) =\sum_{n=N}^{\infty} 2^{-n\alpha} = \frac1{1-2^{-\alpha}}\,2^{-N\alpha}.
\end{equation*}

\noindent Let $i,j\in\{1,\ldots,6\}$. We obtain
\begin{align}\label{twosums}
\sup_{y,y'\in\cal{C}_c^N} |f(y')-f(y)|
& \geqslant \frac1{2}\,\bigl((f(c)-f(c_i))+(f(c)-f(c_j))\bigl)\nonumber\\
& \geqslant \frac1{2}\,\biggl(\,\sum_{n=0}^{N-1}Z_n(c,c_i)+\sum_{n=0}^{N-1}Z_n(c,c_j)\biggl) + \frac{1}{1-2^{-\alpha}}\,2^{-N\alpha}.
\end{align}

\noindent For a fixed hexagon $\cal{C}_c^N$, we claim that there exists a pair of two diametrically opposed vertices $c_i$ and $c_j$ such that the sum of the two first sums above is positive. Indeed, for $n<N$, the center $c$ is included in one or two cells $\cal{C}_c'^n$ of the tessellation of generation $n$ and it can be only in three positions: at the center $c'$ of a cell from $\cal{T}_n$, on one `edge' of the skeleton $\sk_n$ (i.e. a segment between two consecutive vertices $c'_i$ and $c'_{i+1}$ or between the center $c'$ and one of its vertex $c'_i$), or on a `face' (i.e. an open triangle of vertices $c'$, $c'_i$ and $c'_{i+1}$ or equivalently, a connected component of the complementary set of $\sk_n$). Let us denote by $C$, $E$ and $F$ respectively these three positions of $c$. We are interested in the behaviour of the sequence of the positions of $c$ when $n$ goes from $N$ to $0$. In particular, we notice the two following facts.

\begin{itemize}
\item The set of centers of $\cal{T}_{n-1}$ is included in the set of centers of $\cal{T}_n$. Consequently, if $c$ is in position $E$ or $F$ at step $n$, it cannot be in position $C$ at step $(n-1)$.
\item The skeleton $\sk_{n-1}$ is included in the skeleton $\sk_n$. Consequently, if $c$ is in position $F$ at step $n$, it cannot be in position $E$ at step $(n-1)$.
\end{itemize}
Consequently, the sequence of positions when $n$ goes from $N$ to $0$ has to be $(C,\cdots,C)$ or $(C,\cdots,C,E,\cdots,E)$ or $(C,C,\cdots,C,E,\cdots,E,F,\cdots,F)$ (see the first generations on Figure \ref{fig:ECF} below).

\begin{figure}[!h]
    \centering
    \includegraphics[width=13cm]{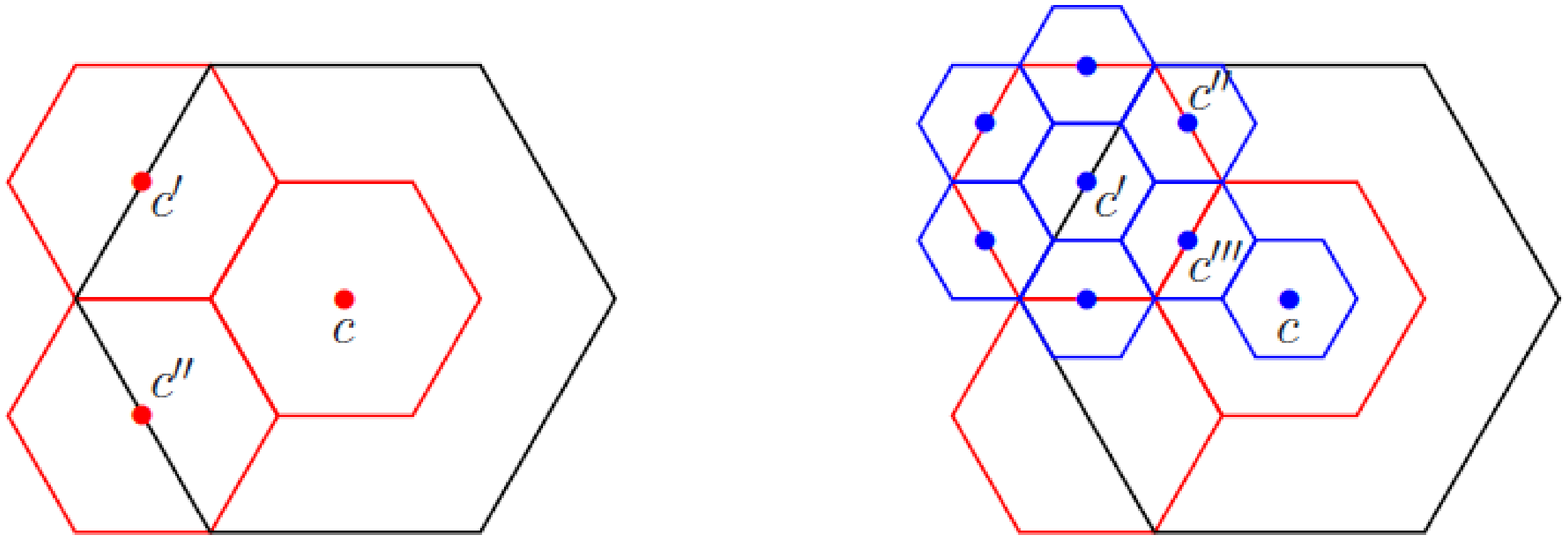}
    \caption{\label{fig:ECF} The only three possibilities for the sequence of the positions of the centers of the cells. On the left $p$ goes from $1$ to $0$: $c$ gives $(\textcolor{red}{C},C)$, $c'$ gives $(\textcolor{red}{C},E)$ and $c''$ gives $(\textcolor{red}{C},E)$. On the right $p$ goes from $2$ to $0$: $c$ gives $(\textcolor{blue}{C},\textcolor{red}{C},C)$, $c'$ gives $(\textcolor{blue}{C},\textcolor{red}{C},E)$, $c''$ gives $(\textcolor{blue}{C},\textcolor{red}{E},E)$ and $c'''$ gives $(\textcolor{blue}{C},\textcolor{red}{E},F)$.}
\end{figure}

\noindent Let us choose now the pair $(c_i,c_j)$ of vertices from $\cal{C}_c^N$ which satisfies that the sum of the first two sums in \eqref{twosums} is non-negative. In the sequel, we will use the following general facts.
\begin{itemize}
\item When $c$ is in position $E$ at step $n$, then there exist exactly two diametrically opposed vertices from $\cal{C}_c^N$ which are on the same edge of $\sk_n$ as $c$.
\item When $c$ is in position $F$ at step $n$, then the whole cell $\cal{C}_c^N$ is included in the same face as $c$.
\end{itemize}
\noindent{\it Case $(C,\cdots,C)$}. When the sequence of positions of $c$ is $(C,\cdots,C)$, we have $\Delta(2^nc)=1\geqslant \Delta(2^nc_i)$ for every $n<N$. Consequently, for any choice of $c_i$ and $c_j$, the sum of the two sums is non-negative.

\pass{\it Case $(C,\cdots,C,E,\cdots,E)$}. Let $n_0$ be the maximal $n<N$ such that $c$ is in position $E$ at step $n$. Let $c_i$ and $c_j$ be the two diametrically opposed vertices associated with $c$ which are on the same edge of the skeleton $\sk_{n_0}$. For every $n_0< n<N$, we have $Z_n(c,c_i)\geqslant 0$ and $Z_n(c,c_j)\geqslant 0$ (see the previous case). If $n\leqslant n_0$, either the edge of $\sk_n$ containing $c$, $c_i$ and $c_j$ is an edge from an hexagon of $\cal{T}_n$ or it is a segment between a center of an hexagon and one of its vertices. In the first case, we have $\Delta(2^nc)=\Delta(2^nc_i)=\Delta(2^nc_j)=0$. In the second case, since $c$ is the midpoint of $[c_i,c_j]$ and the slope of the pyramid is the same above the three points, we have $\Delta(2^nc)-\Delta(2^nc_i)=-(\Delta(2^nc)-\Delta(2^nc_j))$. Consequently, the sum of the two sums in \eqref{twosums} is non-negative.

\pass{\it Case $(C,C,\cdots,C,E,\cdots,E,F,\cdots,F)$}. We define $c_i$ and $c_j$ as in the previous case. The same reasoning as before can be applied for every $n<N$ such that $c$ is in position $C$ or $E$ at step $n$. Let $n<N$ be such that $c$ is in position $F$ at step $n$. Then the slope above the three points $c$, $c_i$ and $c_j$ is constant so the equality $\Delta(2^nc)-\Delta(2^nc_i)=-(\Delta(2^nc)-\Delta(2^nc_j))$ is still valid. Consequently, the sum of the two sums in \eqref{twosums} is non-negative.

\pass Therefore,
\begin{equation*}
\sup_{y,y'\in \cal{C}_c^N}|F(y')-F(y)| \geqslant C'2^{-N\alpha}
\end{equation*}
where $C'$ is a positive constant which only depends on $\alpha$.

\noindent Finally, let $\tau\in(0,1)$, $x\in[0,1]^2$ and $N\geqslant1$ such that $2^{-N}<\tau \leqslant 2^{-(N-1)}$. The ball $B_{\tau}(x)$ contains a cell $\cal{C}_c^{N+1}$ of generation $N+1$, thus
\begin{equation*}
\osc_{\tau}(x)\geqslant \sup_{y,y'\in \cal{C}_c^{N+1}}|F(y')-F(y)| \geqslant C'2^{-(N+1)\alpha}\geqslant C''\tau^{\alpha},
\end{equation*}
which gives the lower bound in \eqref{eqtheo:takaosc}.

\vspace{0.15cm}\noindent Combining (i) and Proposition $11.1$ in \cite{fal03} (see also \cite{dubtri}) we get $\cal{N}(\tau_N)\sim\tau_N^{\alpha-3}$ as $N\to\infty$. The result is then a consequence of \eqref{eq:boxdimbox}.
\end{proof}

\subsection{Takagi-Knopp series generated by a random perturbation of the dyadic mesh}

\pass

\noindent In this subsection, we stray from the Voronoi partition of $\R^D$. An alternative way of randomizing the underlying partition of a Takagi-Knopp series is the following: the sequence of dyadic meshes ${\mathcal D}_n=2^{-n}{\mathbb Z}^D$, $n\in\N$, is kept but each mesh ${\mathcal D}_n$ is translated by a random uniform vector in $2^{-n}(0,1)^D$. In each cube $C_{n,k}$, $k\in\N$, of the translate of ${\mathcal D}_n$, a random uniform `nucleus' $c_{n,k}$ is chosen independently. The associated random pyramidal function $\Delta_n$ is defined so that it is equal to zero on the translate of ${\mathcal D}_n$ and to $1$ on the set $\{c_{n,k},k\in\N\}$. The Takagi-Knopp type series $F_{\alpha,\beta}$, $\alpha,\beta>0$, is then given by a definition very similar to \eqref{def:voronoiseries} with $\lambda=2$. Compared to our Poisson-Voronoi construction, the main advantage of this model is that it preserves the cube structure so that it could be easier to deal with it in dimension two for applicational purpose in image analysis and pixel representation. The essential drawback is that the rigid structure of the mesh prevents us from obtaining an explicit Hausdorff dimension. Still, some results close to those proved in Section \ref{sec:prelim} can be deduced from similar methods. Indeed, if the random oscillation set ${\mathcal O}_{n,N}$, $n,N\in\N$, the increment $Z_n(x,y)$, $x,y\in\R^D$, the density $g_{Z_n}$ of $Z_n(x,y)$ conditionally on $\{x\in{\mathcal O}_{n,N}\}$ and the Lipschitz constant $L(x)$ are defined analogously, then the conclusions of Propositions \ref{prop:size} and \ref{prop:lipconst} are satisfied. Moreover, the point (ii) of Proposition \ref{prop:denshn} is replaced by the following estimate:
\begin{equation}\label{eq:estimzncarre}
\sup_{t\in\R}\,g_{Z_n}(t)\leqslant \frac{C}{\P(x\in {\mathcal O}_{n,n})}\,\lambda^{-\frac{n(\beta-\alpha)}{D}}\sum_{i=1}^D\frac1{|x_i-y_i|}.
\end{equation}
As a consequence, it is possible to derive an analogue of Proposition \ref{prop:estimoscill} and of the upper-bound of the box-dimension in \eqref{eq:maintheo}. This upper-bound is in particular the exact Hausdorff dimension when $D=1$. Indeed, in the linear case, the estimate of $\sup_{t\in\R}\,g_{Z_n}(t)$ in \eqref{eq:estimzncarre} coincides with (ii) of Proposition \ref{prop:denshn}, which implies that the lower-bound can be obtained along the same lines as in Section \ref{sec:proofmain}. We sum up our results in the next proposition, given without a detailed proof.
\begin{prop}\label{prop:laderniere}
Let $F_{\alpha,\beta}$ be the function as above with $0<\alpha\leqslant\beta\leqslant 1$. Then, its graph $\Gamma_{\alpha,\beta}$ satisfies almost surely the following estimates.
  \begin{enumerate}
  \item[(i)] For every $D\geqslant 2$, $\dimbox(\Gamma_{\alpha,\beta})\leqslant D+1-\frac{\alpha}{\beta}$.
  \item[(ii)] When $D=1$, $\dimbox(\Gamma_{\alpha,\beta})= \dimhaus(\Gamma_{\alpha,\beta})=2-\frac{\alpha}{\beta}$.
  \end{enumerate}
\end{prop}
\noindent Unfortunately, the fact that the sum $\sum_{i=1}^D\frac1{|x_i-y_i|}$ is a substitute for the inverse of the Euclidean norm $\|x-y\|^{-1}$ in the estimate \eqref{eq:estimzncarre} makes the proof of the lower-bound of the Hausdorff dimension more intricate for $D\geqslant 2$. This could be eventually considered as an extra-argument in favor of our Poisson-Voronoi construction.

\pass \textsc{Acknowledgement.} We thank an anonymous referee for a careful reading of the original manuscript, resulting in an improved and more accurate exposition.

\bibliographystyle{plain}

\end{document}